\documentclass[11pt,a4paper]{article}

\usepackage{color}
\usepackage{amsthm}
\usepackage{subfigure}
\usepackage{mathtools}
\usepackage{authblk}
\input{Pack.sty}

\newcommand{\fb}[1]{{\color{black}#1}}
\newcommand{\A}[1]{{\color{black}#1}}
\newcommand{\B}[1]{{\color{black}#1}}
\newcommand{\BB}[1]{{\color{black}#1}}
\newcommand{\both}[1]{\color{black}{#1}}
\begin{document}
\title{POD-Galerkin Model Order Reduction for Parametrized Time Dependent Linear Quadratic Optimal Control Problems in Saddle Point Formulation}
\author[$\sharp$]{Maria Strazzullo}
\author[$\sharp$]{Francesco Ballarin}
\author[$\sharp$]{Gianluigi Rozza}
\affil[$\sharp$]{mathlab, Mathematics Area, International School for Advanced Studies (SISSA), Via Bonomea 265, I-34136 Trieste, Italy}
\date{}                     
\setcounter{Maxaffil}{0}
\renewcommand\Affilfont{\itshape\small}
\maketitle

\begin{abstract}
\no In this work we recast parametrized time dependent optimal control problems governed by partial differential equations in a saddle point formulation and we propose reduced order methods as an effective strategy to solve them. Indeed, on one hand parametrized time dependent optimal control is a very powerful mathematical model which is able to describe several physical phenomena; on the other hand, it requires a huge computational effort. Reduced order methods are a suitable approach to have rapid and accurate simulations. We rely \fb{on} POD-Galerkin reduction over the physical and geometrical parameters of the optimality system in a space-time formulation. Our theoretical results and our methodology are tested on two examples: a boundary time dependent optimal control for a Graetz flow and a distributed optimal control governed by time dependent Stokes equations. With these two experiments the convenience of the reduced order modelling is \fb{further extended to} the field of time dependent optimal control.
\end{abstract}

\maketitle

\vspace{.5cm}
\no \textbf{Keywords}:
reduced order methods, proper orthogonal decomposition, time dependent parametrized optimal control problems, time dependent PDEs state equations, saddle point formulation.

\no \textbf{AMS}: 
49J20, 76N25, 35Q35
\section{Introduction}
\label{intro}

Parametrized optimal control problems (OCP($\boldsymbol{\mu}$)s) governed by parametrized partial differential equations (PDE($\boldsymbol{\mu}$)s) \fb{play a ubiquitous role in several applications, yet are} very challenging to analyse theoretically and \fb{simulate} numerically.
In a parametrized setting, where a parameter $\boldsymbol{\mu} \in \mathscr P \subset \mathbb R^d$ could represent physical or geometrical features, \ocp s can be helpful in order to describe and simulate different configurations of several physical and natural phenomena. \\
Indeed, optimal control framework is very versatile and it has been exploited in many contexts and fields: from shape optimization, see e.g. \cite{delfour2011shapes,makinen,mohammadi2010applied}, to fluid dynamics, see e.g. \cite{dede2007optimal,negri2015reduced,optimal,de2007optimal}, from biomedical applications \fb{\cite{Ballarin2017,LassilaManzoniQuarteroniRozza2013a,ZakiaMaria,Zakia}} to environmental ones \cite{quarteroni2005numerical,quarteroni2007reduced,StrazzulloBallarinMosettiRozza2018,ZakiaMaria}.
Even if \ocp s are a very powerful tool widespread in many research fields, \fb{their complexity and computational demanding simulations still limit their applicability}. \fb{Furthermore, the required computational efforts get even larger} if the optimization is constrained to a time dependent
PDE($\bmu$). Surely, time optimization makes the mathematical model more complete and it arises in many applications, see e.g. \cite{HinzeStokes,Iapichino2,leugering2014trends,seymen2014distributed,Stoll1,Stoll}.

\no \fb{The computational effort required for} \ocp s simulations becomes \fb{unbearable} and not viable when the time dependent optimal control depends on a variety of physical and/or geometrical parameters: indeed, in this context, many configurations are studied for several values of $\pmb \mu \in \mathscr P$, increasing the necessary time in order to understand the phenomena behaviour with respect to different parameters. A rapid and suitable approach to manage this drawback \fb{is to} rely on reduced order methods (ROMs), which allow us to solve the parametrized optimality system  in a low dimensional framework, reducing the computational costs, see for example \cite{hesthaven2015certified,prud2002reliable,RozzaHuynhManzoni2013,RozzaHuynhPatera2008}.
Literature is quite complete with respect to the applications of ROM techniques for parametrized steady \ocp s, the interested reader may refer to the following far-from-exhaustive list \cite{bader2016certified,bader2015certified,dede2010reduced,gerner2012certified,Iapichino1,karcher2014certified,karcher2018certified,kunisch2008proper,negri2015reduced,negri2013reduced,quarteroni2007reduced}, where reduced optimal control is treated and analyzed for several state equations and managed with different approaches and methodologies.\\
\fb{We seek therefore to extend} the consolidated knowledge about ROM for steady \ocp s to time dependent \ocp s.
\\

\no Our purpose is to adapt the formulation presented in \cite{negri2015reduced,negri2013reduced} for steady \ocp s to quadratic optimization models constrained to linear time dependent PDE($\bmu$)s. We propose two simple test cases in order to validate our theoretical framework:
\begin{itemize}
\item a time dependent boundary optimal control problem for Graetz flow;
\item a time dependent distributed optimal control problem for Stokes equations.
\end{itemize}
Both the proposed examples have geometrical and physical parametrization.
\\ The main merit of this work is to recast linear quadratic time dependent \ocp s in the very well known and general framework of saddle point problems. To the best of our knowledge, another element of novelty is the numerical simulation \fb{and subsequent reduction} of a time dependent \ocp $\:$ governed by Stokes equations. \both{The proposed techniques have some limitations in practical applications: indeed, we rely on direct solvers for very complex linear systems. This limits our capability to increase the resolutions both in time and space for simulations. This issue could be overcome using more computational resources and through preconditioned Krylov solvers and multigrid approaches: we refer the interested reader to \cite{benzi_golub_liesen_2005,schoberl2007symmetric,Stoll1,Stoll}}.  \BB{We highlight that the proposed iterative approaches may only lighten the computational effort needed for the linear system at hand during the so-called offline phase of the ROM. Due to this, we are aware that the numerical simulations of time dependent \ocp s in the all-at-once saddle point framework could still be unfeasible for actual real-time applications because of the extremely large CPU time required by the offline stage. Nonetheless, this paper aims at showing how ROMs can recover the behaviour of full order simulations with a substantial reduction of computational time. Indeed, we validate our methodology thanks to several experiments for academic test cases. This is only a first step towards the study of new approaches for the application of ROMs to more general and complicated time dependent \ocp s in saddle point formulation}.
\\

\no This work is outlined as follows. In section 2, we rely on theoretical formulation for time dependent optimal control presented in
\cite{hinze2008optimization} and \cite{troltzsch2010optimal}, then, we prove the well-posedness of linear quadratic parabolic time dependent \ocp s in a saddle point framework verifying the standard hypotheses of Brezzi's theorem \cite{boffi2013mixed,Brezzi74}. A brief introduction of the full order Finite Element (FE) approximation precedes the discretization of the optimality system, treated exploiting an all-at-once approach proposed in \cite{Stoll1,Stoll,Yilmaz}. Section 3 introduces the reduced order approximation following \cite{hesthaven2015certified,karcher2014certified} and how \fb{to} apply Proper Orthogonal Decomposition (POD) sampling algorithm for OCP($\boldsymbol{\mu}$)s in saddle point framework with a brief mention of aggregated reduced space strategy used in \cite{dede2010reduced,negri2015reduced,negri2013reduced} and affine decomposition, see e.g. \cite{hesthaven2015certified}. In section 4, numerical results for geometrical and physical parametrization of a boundary optimal control problem for Graetz flow are shown: it is a time dependent version of an \ocp $\:$ presented in \cite{negri2013reduced}. In section 5, we analyze the saddle point structure of time dependent \ocp s governed by Stokes equations. We prove the well-posedness of the mixed problem and we briefly show the discretization techniques used extending the strategies already described in sections 2 and 3. In section 6 we will present a distributed optimal control problem governed by time dependent Stokes equations having consistent results with respect \fb{to} those in \cite{HinzeStokes,Stoll}. Conclusions follow in section 7.

\section{Problem Formulation and Full Order Discretization for Linear Quadratic  Parabolic Time Dependent \ocp s}

In this section parabolic time dependent \ocp s are presented in saddle point formulation. To the best of our knowledge, the saddle point theoretical analytic framework in the context of time dependent \ocp s is a novelty element, even if it is quite a standard approach for stationary linear state equations \cite{karcher2014certified,negri2015reduced,negri2013reduced,rozza2012reduction}. First of all, the well-posedness of the saddle point structure is proved in a space-time formulation. Then, we will introduce the full order discretized problem: the high fidelity approximation is presented following the all-at-once structure exploited in \cite{HinzeStokes,schoberl2007symmetric,Stoll1,Stoll}.

\subsection{Saddle Point Structure of Linear Quadratic Parabolic Time Dependent \ocp s: Theoretical Analysis}
\label{problem}
This section aims at recasting \ocp s governed by linear time dependent PDE($\bmu$)s in a saddle point formulation. In the systems \fb{at hand} the parameter $\pmb \mu \in \mathscr P \subset \mathbb R^p$ could be physical or geometrical. In the following analysis we assume a physical parametrized PDE($\bmu$), but the reached conclusions do not change for the case of geometrical parametrization. \\
Let us consider $\Omega \subset \mathbb R^n$ an open and bounded regular domain and the time interval $\B{(0,T)}$. Let us indicate with $\Gamma_D$ and $\Gamma_N$ the portions of the boundary $\partial \Omega$ where Dirichlet and Neumann boundary conditions are applied, respectively. Let $Y$ and $H$ be two separable Hilbert spaces which verify $Y \hookrightarrow H \hookrightarrow Y\dual$ and  let us consider other two Hilbert spaces $U$ and $Z \supseteq Y$.
Let us define the Hilbert spaces $\Cal Y= L^2(0,T; Y) $, $\Cal Y \dual = L^2(0,T; Y\dual) $, $\Cal U = L^2(0,T; U)$ and $ \Cal Z :=  L^2(0,T; Z) \supseteq \Cal Y$.
Finally, we define the space
$
\Cal Y_t :=
\displaystyle \Big \{
y \in \Cal Y \;\; \text{s.t.} \;\;  \dt{y} \in \Cal Y\dual
\Big \}.
$
\no The problem to be solved is the following: for a given $\pmb \mu \in \mathscr P$ find the \fb{pair}
$(y,u) \in \fb{\Cal Y_t} \times \Cal U$ which solves
\begin{equation}
\label{J_functional}
\min_{(y,u) \in \Cal Y_t \times \Cal U}
\frac{1}{2} \intTimeSpace{(y - y_d(\pmb \mu))^2} 
+ \alf \intTimeSpace {u^2}
\end{equation}
under a \fb{constraint} of the form
\begin{equation}
\label{eq_time_strong}
\begin{cases}
\vspace{.1cm}
\displaystyle \dt {y} + \Cal {D}_a(\pmb \mu)y = \Cal{D}_c(\pmb \mu)u + f& \text{in } \Omega \times \B{(0,T)}, \\
 \displaystyle \dn{y} = 0 & \text{on  } \Gamma_N \times \B{(0,T)}, \\
\displaystyle y = g & \text{on  } \Gamma_D \times \B{(0,T)}, \\
y(0) = y_0 & \text{in } \Omega,
\end{cases}
\end{equation}
where $\alpha > 0 $ is a fixed constant, $\Cal D_a(\bmu)$ and $\Cal D_c(\bmu)$ are general parametrized differential operators associated to state and control, respectively, and $y_d(\pmb \mu)$ is an observation in the space $\Cal Z$. Furthermore, \fb{for the sake of simplicity} we assume that $\norm{y_0}^2_{H} = 0$ and $\Cal Z = \Cal Y_t$. 
Problems of this kind have been covered extensively in classical references as \cite{lions1971,troltzsch2010optimal}, for example, where the topic of optimal control is treated in its entirety. Basically, we are adding a parametrization framework to this very versatile mathematical model, in order to study different physical and/or geometrical configurations. Indeed, we underline that we are actually looking for solutions which are parameter dependent, but for the sake of notation we write $y := y(\pmb \mu)$ and $u := u(\pmb \mu)$: depending on the case, we will omit the parameter dependence from now on.


\no Moreover, we suppose that for every $\bmu \in \mathscr P$ and every control variable $u$ there exists an unique solution $y:= y(u)$.
The time dependent state equation \eqref{eq_time_strong} can be expressed in the following weak formulation:

\hspace{-1cm}\begin{equation}
\label{eq_time_weak}
\begin{cases}
\displaystyle \intTime{ \Big \la \dt y, q \Big \ra_{Y\dual, Y}}
+ \intTime {a (y, q; \bmu)} =
\intTime {c(u, q; \bmu)}
+\intTime{ \la G(\bmu), q \ra_{Y\dual, Y}} & \forall q \in \Cal Y_t, \\
y(0) = y_0  & \text{in } \Omega,
\end{cases}
\end{equation}
where $a \goesto{\Cal Y_t}{\Cal  Y_t}{\mathbb R}$ and $c \goesto{\Cal U}{\Cal Y_t}{\mathbb R}$ are the bilinear forms \fb{associated to $\Cal {D}_a(\pmb \mu)$} and $\Cal {D}_c(\pmb \mu)$, respectively, while $G \in \Cal Y\dual $ \fb{collects} forcing and boundary terms \fb{of the state equation}.
In our applications \fb{it will always be the case} $\displaystyle c(u,q;\bmu) = \intSpace{uq} $.
\fb{We further remark that we consider $q \in \Cal Y_t$ rather than $q \in \Cal Y$ as it will be convenient to restrict $q$ to $\Cal Y_t$ for a proper definition of the adjoint variable.}
Let us define the following bilinear forms
\begin{align}
m \goesto {\Cal Y_t}{\Cal Y_t}{\mathbb R}  \hspace{1cm} & m(y,z) =  \intTimeSpace{yz},\\
n \goesto {\Cal U}{\Cal U}{\mathbb R} \hspace{1cm} & n(u,v) = \intTimeSpace{uv},
\end{align}
which represent the objective for the state variable and for the control variable in the whole time interval, respectively.
Thus, let us define the functional
\begin{equation}
\label{eq_functional}
J(y,u; \bmu) = \frac{1}{2} m(y - y_d(\bmu), y - y_d(\bmu))
+ \alf n(u,u).
\end{equation}
The \ocp   $\:$ \fb{reads} as follows:  given $\bmu \in \mathscr P$, find the solution of

\begin{equation}
\label{general_problem}
\min_{(y, u) \in \Cal Y_t \times \Cal U} J(y, u; \bmu) \spazio \text{such that \eqref{eq_time_weak} is satisfied}.
\end{equation}
In order to set \fb{the} problem in a mixed formulation, we need to define the state-control product space $\Cal X = \Cal Y_t \times \Cal U$. Given $x = (y,u)$ and $w = (z,v)$ elements of $\Cal X$, the scalar product of $\Cal X$ is defined by
$
(x,w)_{\Cal X} = (y,z)_{\Cal Y_t} + (u,v)_{\Cal U},
$
that induces the norm $\norm{\cdot}_{\Cal X}$, since
$(\cdot, \cdot)_{\Cal Y_t}$ and $(\cdot, \cdot)_{\Cal U}$ define the following norms, respectively:
\begin{align*}
\norm{y}_{\Cal Y_t}^2 & = \norm{y}_{\Cal Y}^2 +\parnorm{\dt y}_{\Cal Y\dual}^2
=  \intTime{\norm{y}_Y^2} +  \intTime{\parnorm{\dt y}_{Y\dual}^2}, \qquad
 \norm{u}_{\Cal U}^2 = \intTime{\norm{u}_U^2}.
\end{align*}
Moreover, in the following, we need the forms:
\begin{align*}
\Cal A \goesto {\Cal X}{\Cal X}{\mathbb R}\A{,} & \spazio
\Cal A(x,w) = m(y,z) +\alpha n(u,v)\A{,} &
\quad \forall x, w \in \Cal X, \\
\Cal B \goesto {\Cal X}{\Cal Y_t}{\mathbb R}\A{,} & \spazio
\Cal B(w, q; \bmu) = \intTime { \Big \la \dt z , q \Big \ra } + \intTime{a(z, q; \bmu)} -
\intTime{c(v, q; \bmu)}\A{,}&
\quad \forall w \in \Cal X, \forall q \in \Cal Y_t,\\
F(\bmu) \in \Cal X\dual\A{,} & \spazio
\intTime{\la F(\bmu), w \ra} = m(y_d(\bmu),z)\A{.} &
\quad \forall w\in \Cal X. \\\vspace{-7mm}
\end{align*}

\noindent We underline that for the analysis and for sake of notation, the problem we propose has a distributed control, but it can be extended also for boundary control if one defines:
$$\norm{u}_{\Cal U}^2 = \intTime{\norm{u}_{U(\Gamma_c)}^2},$$
where $\Gamma_c \subseteq \partial \Omega$ is the boundary portion where the control is applied. An application of boundary control will be presented in Section \ref{ADR_OCP}. \\
As we did for the state and the control variables, for the sake of notation we define $p := p(\bmu)$. In order to build the optimality system, first of all we construct the following Lagrangian functional

\begin{equation}
\label{functional}
\Lg (y,u,p; \bmu) = \Cal J((y,u); \bmu) + \Cal B((y,u), p; \bmu)
+ \intTime{\la G(\bmu), p \ra},
\end{equation}
where
\begin{equation}
\label{cal_functional}
\Cal J((y,u), \bmu) = \half  \Cal A (x,x; \pmb \mu)  - \intTime{\la F(\bmu), x \ra},
\end{equation}
recalling that $ x= (y,u)$.
Then, we perform a differentiation by the adjoint variable $p$, the state variable $y$ and the control $u$.
The minimization of \eqref{functional} is equivalent to find the solution of the following system: given $\pmb \mu \in \mathscr P$, find $(x, p) \in \Cal X \times \Cal Y_t$ such that
\begin{equation}
\label{optimality_system}
\begin{cases}
D_y\Lg(y, u, p; \pmb \mu)[z] = 0 & \forall z \in \Cal Y_t,\\
D_u\Lg(y, u, p; \pmb \mu)[v] = 0 & \forall v \in \Cal U,\\
D_p\Lg(y, u, p; \pmb \mu) [q]= 0 & \forall q \in \Cal Y_t.\\
\end{cases}
\end{equation}

\no The framework introduced is totally general and it also holds for nonlinear time dependent problems.
\\ We now focus on the  case of \tbf{parabolic linear governing equations}. In this case, the system of equations related to \eqref{optimality_system} could be written  in the following form:

\begin{equation}
\label{linear_optimality system}
\begin{cases}
\Cal A (x,w)  +
	\displaystyle \Cal B(w, p; \bmu) =\intTime{ \la F(\pmb \mu), w \ra} & \forall w \in \Cal X, \\
	\displaystyle  \Cal B(x, q; \bmu) = \intTime{\la G(\pmb \mu), q \ra} & \forall q \in \Cal Y_t.\\
\end{cases}
\end{equation}

 \no As one can see, the classical structure of saddle point formulation that characterizes steady linear quadratic \ocp s is preserved also in the linear time dependent case. We now want to provide the \fb{well-posedness} of the problem \eqref{linear_optimality system} through the fulfillment of  \cite[\A{Proposotion} 1.7]{bochev2009least}, a \A{Proposotion} based on Brezzi's theorem \cite{boffi2013mixed,Brezzi74}:
\begin{prop}
\label{prop}
Assume that the Brezzi's theorem is verified, i.e.
\begin{enumerate}
\item $\Cal A \cd$ is continuous and weakly coercive on the kernel of $\Cal B(\cdot, \cdot; \pmb \mu)$, that we indicate with
$\Cal X_0 \subset \Cal X$, i.e.  for every $x, w \in \Cal X_0$ it holds:
\begin{equation}
\label{weak_coercivity}
\inf_{w \in \Cal X_0 \setminus \{0\}} \sup_{x \in \Cal X_0 \setminus \{0\}} \frac{\Cal A (x,w)}{\norm{x}_{\Cal X_0}\norm{w}_{\Cal X_0}} > 0
\quad \text{ and } \quad
\inf_{x \in \Cal X_0 \setminus \{0\}} \sup_{w \in \Cal X_0 \setminus \{0\}} \frac{\Cal A (x,w)}{\norm{x}_{\Cal X_0}\norm{w}_{\Cal X_0}} > 0.
\end{equation}
\item $\Cal B (\cdot, \cdot; \pmb \mu)$ is continuous and satisfies the \textit{inf-sup condition}, i.e. for every $x \in \Cal X$ and $q \in \Cal Y_t$ the following inequality is verified:
\begin{equation}
\label{infsup_2}
\inf_{q \in \Cal Y_t \setminus \{0\}} \sup_{x \in \Cal X  \setminus \{0\}} \frac{\Cal B(x, q; \bmu)}{\norm{x}_{\Cal X}\norm{q}_{\Cal Y_t}} = \beta(\bmu) > 0.
\end{equation}
\end{enumerate}
Let us assume further that the bilinear form $\Cal A\cd$ is symmetric, non-negative, and coercive on $\Cal X_0$, then the minimization of the functional \eqref{cal_functional} constrained to equation \eqref{eq_time_strong} and the resolution of the saddle point problem \eqref{linear_optimality system} are equivalent.
\end{prop}
In order to prove that the hypotheses over $\Cal A \cd$ and $\Cal B (\cdot, \cdot; \pmb \mu)$ hold for \ocp s governed by linear time dependent state equations, first of all, we require that:
\begin{Assum} the bilinear forms $c (\cdot ,\cdot; \pmb \mu), \; a (\cdot ,\cdot; \pmb \mu),
\; n (\cdot ,\cdot)$ and $m (\cdot ,\cdot)$ verify the properties
\label{ass_1}
\begin{enumerate}[(i)]
\item  $|c (u,q; \pmb \mu)| \leq c_c(\bmu) \norm{u}_{U} \norm{q}_{Y} \qquad \forall u \in \Cal U
\text{ and } \forall q \in \Cal Y_t$;
\item $|a(y,q; \pmb \mu)| \leq c_a(\bmu) \norm{y}_{Y} \norm{q}_{Y} \qquad \forall y,q \in \Cal Y_t$;
\item $a(y,y; \pmb \mu) \geq M_a(\bmu) \norm{y}_{Y}^2 \qquad \forall y \in \Cal Y_t$;
\item
\label{n}
$n \cd$ is symmetric, continuous and such that $n(u,u) \geq \gamma_n \norm{u}^2_{\Cal U}$;
\item
\label{m}
$m \cd$ is symmetric, continuous and positive definite.
\end{enumerate}
\end{Assum}
%
 \no Furthermore, we will exploit the following inequalities in order to assert the inequalities \eqref{weak_coercivity} and  \eqref{infsup_2}:
\begin{itemize}

\item[I.]  by definition, for every $y \in \Cal Y_t$ and $u \in \Cal U$  it holds:
\begin{equation}
\displaystyle \parnorm{\dt y}_{\Cal Y\dual} \leq \norm{y}_{\Cal Y_t}, \quad \norm{y}_{\Cal Y} \leq \norm{y}_{\Cal Y_t}, \quad \norm{y}_{\Cal Y_t} \leq \norm{x}_{\Cal X}  \quad \text{and} \quad \norm{u}_{\Cal U} \leq \norm{x}_{\Cal X};
\end{equation}
\item[II.] for $y$ solution of a parabolic PDE($\pmb \mu$) with forcing term $f$ and $y(0)=y_0$, there exists $k(\bmu)>0$ such that:
\begin{equation}
\label{standard_inequality}
\norm{y}_{\Cal Y_t} \leq k(\bmu)(\norm{f}_{L^2(0,T; Y)} + \norm{y_0}_Y).
\end{equation}
\end{itemize}
Two more ingredients used to guarantee the \emph{inf-sup condition} \eqref{infsup_2} are the following two lemmas.

\begin{lemma}
\label{lemma_1_6}
Let $q$ be a function in $\Cal Y_t$, then the following inequality holds:
$$
\frac{ \norm{q}_{\Cal Y}^2 }{\norm{q}_{\Cal Y_t}^2} \geq \frac{1}{6}.
 $$
\end{lemma}
\begin{proof}
We divide \fb{the} proof in two cases: first of all, let \fb{us} assume
$\displaystyle \parnorm{\dt{q}}_{\Cal Y\dual} \leq \norm{q}_{\Cal Y}$ then \\ $\norm{q}_{\Cal Y_t}^2 \leq 2 \norm{q}_{\Cal Y}^2$. This leads to the following inequality:

$$ \frac{ \norm{q}_{\Cal Y}^2 }{\norm{q}_{\Cal Y_t}^2}
\geq \frac{\norm{q}_{\Cal Y}^2}{2\norm{q}_{\Cal Y}^2} =
\frac{1}{2} > \frac{1}{6}.
$$
We now want to prove a similar inequality for $q \in \Cal Y_t$ such that
$\displaystyle \parnorm{\dt{q}}_{\Cal Y\dual} > \norm{q}_{\Cal Y}$.  This assumption allows us to assert that:
\begin{equation}
\label{in1}
 2 \norm{q}_{\Cal Y}^2 < \norm{q}_{\Cal Y_t}^2  \spazio \Rightarrow \spazio
\frac{1}{2 \norm{q}_{\Cal Y}^2} > \frac{1}{\norm{q}_{\Cal Y_t}^2},
\end{equation}
and, by definition of $\norm{q}_{\Cal Y_t}$,
\begin{equation}
\label{in2}
\norm{q}_{\Cal Y_t}^2 + 2\parnorm{\dt{q}}_{\Cal Y \dual}^2
\geq 3\parnorm{\dt{q}}_{\Cal Y \dual}^2
 \Rightarrow - \displaystyle \frac{1}{\norm{q}_{\Cal Y_t}^2 + 2 \displaystyle \parnorm{\dt{q}}_{\Cal Y \dual}^2
} \geq - \frac{1}{3\displaystyle  \parnorm{\dt{q}}_{\Cal Y \dual}^2}.
\end{equation}
Then, we can prove that
\begin{align*}
\frac{ \norm{q}_{\Cal Y}^2 }{ \norm{q}_{\Cal Y_t}^2}
 & =  \frac{ 2 \norm{q}_{\Cal Y}^2 }{2\norm{q}_{\Cal Y}^2 + 2 \displaystyle \parnorm{\dt{q}}_{\Cal Y\dual}^2  }  \underbrace{\geq}_{\A{\text{for } \eqref{in1}}}
\frac{ 2 \norm{q}_{\Cal Y}^2 }{ \norm{q}_{\Cal Y_t}^2 + 2 \displaystyle \parnorm{\dt{q}}_{\Cal Y\dual}^2  } \\
& = \frac{ 2 \norm{q}_{\Cal Y}^2 + \displaystyle \parnorm{\dt{q}}_{\Cal Y \dual}^2 -  \parnorm{\dt{q}}_{\Cal Y \dual}^2  }{\norm{q}_{\Cal Y_t}^2 + 2 \displaystyle \parnorm{\dt{q}}_{\Cal Y\dual}^2  } \\
& \geq \frac{\min \{1,2 \} \norm{q}_{\Cal Y_t}^2}{\max \{ 1, 2 \} \norm{q}_{\Cal Y_t}^2} - \frac{\displaystyle \parnorm{\dt{q}}_{\Cal Y \dual}^2}{\norm{q}_{\Cal Y_t}^2 + 2 \displaystyle \parnorm{\dt{q}}_{\Cal Y\dual}^2 } \\
& \underbrace{\geq}_{\A{\text{for } \eqref{in2}}} \frac{1}{2}
- \frac{\displaystyle \parnorm{\dt{q}}_{\Cal Y\dual}^2}{3 \displaystyle \parnorm{\dt{q}}_{\Cal Y\dual}^2} = \frac{1}{2} - \frac{1}{3} = \frac{1}{6}.
\end{align*}
\cvd
\end{proof}
\noindent Let us now prove a second lemma needed to show the well-posedness of problem \eqref{linear_optimality system}.
\begin{lemma}
\label{lemma_existence}
Given a function $v \in \Cal Y_t$, there exists $\bar y \in \Cal Y_t$ which verifies:
\begin{equation}
\label{useful_eq}
 \intTime { \Big \la \dt {\bar y} , q \Big \ra }
+ \intTime{a(\bar y, q; \bmu)} = \intTime{a(v, q; \bmu)} \qquad  \forall q \in \Cal Y_t,
\end{equation}
with $\bar{y}(0)= 0$. Moreover, there exists a positive constants $\bar k(\bmu)$ such that the following inequality holds:
\begin{equation}
\norm{\bar y}_{\Cal Y_t} \leq \bar k(\bmu)\norm{v}_{\Cal Y}.
\end{equation}
\end{lemma}

\begin{proof}
The existence of the solution $\bar y$ is actually proposed in the proof of property (A.3) of the \A{Theorem} 5.1 in \cite{schwab2009space}, where the existence of $\bar y \in \Cal Y_t$ is guaranteed for a given $v \in \Cal Y \supset \Cal Y_t$ and for every initial condition. \\
Now, for $t \in \B{(0,T)}$ let us consider the linear operator $\Cal {D}_a(\bmu): Y \rightarrow Y\dual$ defined by
$\la \Cal {D}_a v, q \ra_{H\dual, H}$. We \fb{define $C_{\Cal {D}_a}(\bmu) = \norm{\Cal {D}_a}_{H\dual}$ which is finite since $\Cal {D}_a$ is a continuous operator}.
Furthermore, since $\bar{y}$ verifies \eqref{useful_eq}, from the standard inequality \eqref{standard_inequality} we can derive the inequality
\begin{equation*}
\norm{\bar y}_{\Cal Y_t} \leq k(\bmu)\norm{\Cal {D}_a v}_{L^2(0,T;H)} \leq k(\bmu)C_{\Cal {D}_a}\norm{v}_{L^2(0,T; H)}.
\end{equation*}
Since $Y \hookrightarrow H$, it holds $\norm{q}_{H} \leq \bar C \norm{q}_{Y}$, then, setting
$\bar k(\bmu) = k(\bmu)C_{\Cal {D}_a} \bar C$, we can prove the thesis and
\begin{equation}
\norm{\bar y}_{\Cal Y_t} \leq \bar k(\bmu)\norm{v}_{\Cal Y}.
\end{equation}

\end{proof}
\cvd \\
\no We now have all the ingredients which will help us proving the well-posedness of the saddle point system \eqref{linear_optimality system}. Indeed, the following theorem provides conditions \eqref{weak_coercivity} and \eqref{infsup_2} for linear parabolic time dependent \ocp s, i.e. the existence and the uniqueness of an optimal solution for the minimization problem defined by \eqref{linear_optimality system}.
\begin{theorem}
\label{Brezzi_continuous}
The saddle point problem \eqref{linear_optimality system} satisfies the hypotheses of \A{Proposition} \ref{prop} under the \A{Assumptions} \ref{ass_1}, i.e. it has a unique optimal solution.
\end{theorem}

\begin{proof}
Let us consider the continuity of $\Cal A \cd$.
\begin{align*}
|\Cal A(x, w)|  & \leq 
\norm {y}_{\Cal Y_t} \norm {z}_{\Cal Y_t} + \alpha  \norm {u}_{\Cal U}\norm {v}_{\Cal U}  \\
	& \leq 
\max\{1, \alpha\}\norm {x}_{\Cal X}\norm {w}_{\Cal X}.  \\
\vspace{-0.7cm}
\end{align*}
\fb{Indeed, the above inequality follows from continuity of the bilinear forms $m \cd$ and $n \cd$, which can be shown as follows:}
\begin{align*}
\Big | \intTimeSpace{yz} \Big |
& \underbrace{\leq}_{\text{Cauchy}}
 \Big ( \intTime{{\norm y}_Y  \norm z_Y }\Big) \underbrace{\leq}_{\text{Holder}} 
\sqrt{\intTime{\norm y_Y^2}} \sqrt{\intTime{\norm z_Y^2}}=
 \norm y_{\Cal Y} \norm z_{\Cal Y}  \\
					&\hspace{1.7mm}\underbrace{\leq}_{\text{II}}  
\norm y_{\Cal Y_t} \norm z_{\Cal Y_t}.
\end{align*}
The same argument can be used for $n(\cdot, \cdot)$, since
\begin{align*}
\Big |\intTimeSpace{uv} \Big | & \underbrace{\leq}_{\text{Cauchy}} \Big ( \intTime{{\norm u}_U  \norm v_U }\Big)
					 \underbrace{\leq}_{\text{Holder}} \sqrt{\intTime{\norm u_U^2}}\sqrt{\intTime{\norm v_U^2}}
					 \fb{=} \norm u_{\Cal U} \norm v_{\Cal U}.
\end{align*}
Thanks to the hypothesis \eqref{n} and \eqref{m}, $\Cal A(\cdot, \cdot)$ is symmetric,  positive definite and continuous.
\\ We can now prove the coercivity of $\Cal A$ on $\Cal X_0$. If $x \in \Cal X_0$, then it holds
 $\displaystyle \intTime { \Big \la \dt y , q \Big \ra }
+ a(y, q; \bmu) = c(u, q; \bmu)$ and then $\norm{y}_{\Cal Y_t} \leq k(\bmu)( \norm{u}_{\Cal U}+ \norm{y_0}_H) = k(\bmu)\norm{u}_{\Cal U}$ by assumption over the initial condition of $y$. Then, it holds:
\begin{align*}
\Cal A (x, x)
& =
m(y,y) + \alpha n(u,u) \geq \norm{y}_{\Cal Y}^2
+ \alf \norm{u}_{\Cal U}^2 + \alf \norm{u}_{\Cal U}^2  \\
& \geq \frac{\alpha}{2k(\bmu)^2}\norm{y}_{\Cal Y_t}^2 + \alf \norm{u}_{\Cal U}^2
\geq \min
\Big \{ \frac{\alpha}{2k(\bmu)^2}, \frac{\alpha}{2} \Big \} \norm{x}_{\Cal X}^2.
\end{align*}
Let us prove on the continuity of $\Cal B(\cdot, \cdot; \pmb \mu)$.
Exploiting the continuity of $a(\cdot, \cdot; \bmu)$ and $c (\cdot, \cdot; \bmu)$, it holds:
\begin{align*}
| \Cal B (x, q; \bmu) |
& \leq \intTime{\Big | \Big \la \dt y, q\Big \ra \Big | } +
\intTime{c_a(\bmu) \norm{y}_{Y} \norm{q}_{Y}}
+ \intTime{c_c(\bmu) \norm{u}_{U} \norm{q}_{Y}} \\
& \underbrace{\leq}_{\text{Cauchy + H\"older}} \parnorm{\dt y}_{\Cal Y\dual}\norm{q}_{\Cal Y_t}
+ c_a(\bmu) \norm{y}_{\Cal Y_t} \norm{q}_{\Cal Y_t}
+ c_c(\bmu) \norm{u}_{\Cal U} \norm{q}_{\Cal Y_t} \\
& \underbrace{\leq}_{\text{I} }\max\{1, c_a(\bmu), c_c(\bmu) \}\norm{x}_{\Cal X} \norm{q}_{\Cal Y_t},
\end{align*}
where $c_a(\bmu)$ and $c_c(\bmu)$ are the continuity constants of $a (\cdot, \cdot; \pmb \mu)$ and $c (\cdot, \cdot; \pmb \mu)$. respectively.
\\Now we focus on the fulfillment of the \textit{inf-sup condition} for the bilinear form $ \Cal B (\cdot, \cdot; \pmb \mu)$.  First of all, let us consider $q \in \Cal Y_t$ and $\bar y \in \Cal Y_t$ the solution of the problem \eqref{useful_eq} presented in Lemma \ref{lemma_existence} with $v \equiv q$.
Then,
\begin{align*}
\sup_{0 \neq x \in \Cal X } \frac{\Cal B(x, q; \bmu)}{\norm{x}_{\Cal X}\norm{q}_{\Cal Y_t}}
& = \sup_{0 \neq (y,u)} \frac{ \displaystyle \intTime { \Big \la \dt y , q \Big \ra }
+ \intTime{a(y, q; \bmu)} - \intTime{c(u, q; \bmu)}}{\norm{x}_{\Cal X}\norm{q}_{\Cal Y_t}} \\
& \underbrace{\geq}_{x = (\bar y,0)}
 \frac{\displaystyle \intTime { \Big \la \dt {\bar y} , q \Big \ra }
+ \intTime{ a(\bar y, q; \bmu)}}{\norm{\bar y}_{\Cal Y_t} \norm{q}_{\Cal Y_t} }
\displaystyle \geq  \frac{\displaystyle \intTime{a(q,q; \pmb \mu)} }{ \norm{\bar y}_{\Cal Y_t}\norm{q}_{\Cal Y_t}} \\
& \underbrace{\geq}_{\text{Lemma } \ref{lemma_existence}}
\frac{M_a(\bmu) \norm{q}_{\Cal Y}^2}{\bar k(\bmu)\norm{q}_{\Cal Y_t}^2}
 \underbrace{\geq}_{\text{Lemma } \ref{lemma_1_6}} \frac{M_a(\bmu)}{6 \bar k(\bmu)} > 0.
\end{align*}
Since we have proved the inequality for all $q \in \Cal Y_t$, it holds:
\[
\inf_{0 \neq q \in \Cal Y_t} \sup_{0 \neq x \in \Cal X } \frac{\Cal B(x, q; \bmu)}{\norm{x}_{\Cal X}\norm{q}_{\Cal Y_t}} \geq  \frac{M_a(\bmu)}{6 \bar k(\bmu)} : = \beta(\bmu) > 0.
\]
\cvd
\end{proof}
\no The theorem just presented guarantees the existence and uniqueness of an optimal solution for time dependent \ocp s.
%



\vspace{2mm}\no We now propose a brief overview on the FE approach, that we exploit as full order approximation in order to build \fb{the} reduced order model.

\no For the spatial discretization, first of all we define the triangulation $\Cal T \disc$ of $\Omega$. We can now define discretized spatial spaces as $Y \disc = Y \cap\mathscr X^{\Cal N}_r$
and $U \disc = U  \cap \mathscr X^{\Cal N}_r$, where
$$
 \mathscr X^{\Cal N}_r = \{ v \disc \in C^0(\overline \Omega) \; : \; v \disc |_{K} \in \mbb P^r, \; \; \forall K \in \Cal T \disc \}.
$$
The space $\mbb P^r$ is the space of all the polynomials of degree at most equal to $r$ and $K$ is a triangular element of $\Cal T^{\Cal N}$. Then, the function spaces considered are: $\Cal Y \disc= L^2(0,T; Y\disc) $, ${ \Cal Y\dual}^{\Cal N} = L^2(0,T; {Y\dual}^{\Cal N})$ and $\Cal U \disc = L^2(0,T; U\disc)$. The full order \fb{state variable} will be considered in
$
\Cal Y_{t}^{\Cal N} :=
\displaystyle \Big \{
y \in \Cal Y\disc \;\; s.t. \;\;  \dt{y} \in { \Cal Y\dual}^{\Cal N}
\Big \}.
$
\no As we did in section \ref{problem}, for the mixed formulation, we exploit the product space $\Cal X^{\Cal N} = \Cal Y_t \disc \times \Cal U \disc \subset \Cal X$. The Galerkin FE discretization of the saddle point problem \eqref{linear_optimality system} reads as follows: given $\boldsymbol{\mu} \in \mathscr P$, find $(x^{\Cal N}, p^{\Cal N}) \in \Cal X^{\Cal N} \times
\Cal Y^ {\Cal N}$ such that

\begin{equation}
\label{FEOCP}
\begin{cases}
 \Cal A(x^{\Cal N},w\disc) + \Cal B (w\disc,p\disc; \boldsymbol{\mu}) =
\displaystyle
\intTime{\la F(\boldsymbol{\mu}), w\disc \ra} & \forall v\disc \in \Cal X\disc, \\
\Cal B(x\disc,q\disc; \boldsymbol{\mu})=
\displaystyle
\intTime{\la G(\boldsymbol{\mu}), q\disc \ra} & \forall q\disc \in \Cal Y_t \disc,
\end{cases}
\end{equation}
where $x \disc := x \disc (\bmu)$ and $p \disc := p \disc (\bmu)$ are parameter dependent solutions.
Following the same strategies used in \A{Section} \ref{problem}, it can be shown that Brezzi's \A{Theorem} \cite{boffi2013mixed,Brezzi74} holds at the discrete level. Moreover one can prove the coercivity of the form $\Cal A \cd$: then, the minimization problem \eqref{FEOCP} is well-posed \cite{bochev2009least}.

\no In this section we introduced the general formulation for linear time dependent \ocp s and we recast it  in a saddle point system of the form \eqref{linear_optimality system}. Then, we provided a proof for \A{Theorem} \ref{Brezzi_continuous} which led to a unique optimal solution, after having  introduced all the assumptions and lemmas needed in order to have \A{Proposition} \ref{prop} verified.
Furthermore, we introduced FE element approximation as our full order discretization technique. \\
\no In the next section, the full order problem will be shown in space-time formulation \fb{following} the approaches already presented for parabolic PDEs($\pmb \mu$) in \cite{Glas2017,urban2012new,yano2014space,yano2014space1}.

\subsection{Algebraic System and All-at-Once Approach}
\label{Algebraic_parabolic}
In this section we introduce a discretized version of time dependent \ocp s. Practically, we follow the all-at-once space-time discretization already proposed in \cite{HinzeStokes,Stoll1,Stoll}. \\
We now want to consider the full order discretization of the \ocp $\;$introduced in \eqref{general_problem}. Let us suppose to have discretized the spaces considered in \A{Section} \ref{problem} with FE technique in order to solve the full order optimality system \eqref{FEOCP}. We aim at showing how the saddle point structure is reflected also in the algebraic formulation of the problem. We recall that the solution $(y,u,p)$ solves the following optimality system
\begin{equation}
\label{strong_form_optimality_system}
\begin{cases}
\displaystyle y  - \dt{p} + \Cal D\dual_a(p) = 
y_d & \text{ in } \Omega \times \B{(0,T)}, \\
\alpha u - \Cal D_c^\star( p) = 0 & \text{ in } \Omega \times \B{(0,T)}, \\
\displaystyle \dt{y} + \Cal D_a (y) - \Cal D_c (u) = g & \text{ in } \Omega \times \B{(0,T)}, \\
y(0) = y_0  & \text{ in } \Omega , \\
p(T) = 0 & \text{ in } \Omega, \\
\text{boundary conditions} & \text{ on $\partial{\Omega} \times \B{(0,T)},$}
\end{cases}
\end{equation}
where $\Cal D_a$, $\Cal D_c$, $\Cal D_a^\star$ and $\Cal D_c^\star$  are the differential operators associated to $a (\cdot, \cdot; \pmb \mu)$ and $c (\cdot, \cdot; \pmb \mu)$ and their adjoint bilinear form, respectively.
\B{The time integral} has been approximated exploiting the composite rectangle quadrature rule formula. The time interval is divided in $N_t$ sub-intervals of length $\Delta t$.
\begin{remark}
\label{time_int}
{The applied time discretization is actually equivalent to a classical implicit Euler approach \cite{eriksson1987error,yano2014space}}.
We underline that the state equation is discretized forward in time with a backward Euler method. The adjoint equation will be discretized backward in time using the forward Euler method, \B{which is equivalent to backward Euler with respect to time $T-t$, for $t \in (0,T)$}.
\end{remark}

\no Let us begin our analysis from the state equation. Let us define $\pmb y = [y_1, \dots, y_{N_t}]^T$ and $\pmb u = [u_1, \dots, u_{N_t}]^T$ and $\pmb p = [p_1, \dots, p_{N_t}]^T$, \A{where $y_i \in Y\disc$, $u_i \in U\disc$ and $p_i \in Y \disc$ for $1 \leq i \leq N_t$ are the vectors of all the discrete variables at each time step}. With $y_i, u_i$ and $p_i$ we indicate \A{the row vectors} \A{containing} the coefficients of the FE discretization for state, control and adjoint, respectively. \\The vector representing the initial condition for the state variable is
$\pmb y_0 = [y_0, 0, \dots, 0]^T$. The vector $\pmb g = [g_1, \dots, g_{N_t}]^T$ corresponds to the forcing term. Finally, the vector
$\pmb y_d = [y_{{d}_1}, \dots, y_{{d}_{N_t}}]^T$ is our discretized desired state. \\
Let $D_a$ and $D_c$ be the \A{matrices} associated to the bilinear forms applied to the basis functions of the FE spaces ${{Y}}^{\Cal N} := \text{span} \{ \phi_i, 1 \leq i \leq \Cal N \}$ and
${{U}}^{\Cal N} := \text{span} \{ \psi_i, 1 \leq i \leq \Cal N \} $, i.e.
$D_{a_{ij}} = a(\phi_i, \phi_j; \bmu)$ and $\fb{D_{c_{ij}}} = c(\psi_{i}, \phi_j; \bmu)$ for $i,j = 1, \dots, \Cal N$, respectively. For the sake of notation, the dependence of the matrix \A{on} the parameter $\bmu$ is \A{dropped}.
\\The state equation to be solved is
\[
My_k + \Delta t D_a y_k - \Delta t D_c u_k= My_{k-1}  + g_k \Delta t
\spazio \B{ \text{for } k \in \{ 1, 2, \dots, N_t\}},
\]
where $M$ is the mass matrix relative to the FE discretization.

\no In order to solve the system in an all-at-once approach
we can write:
\[
\hspace{-1.8cm}
\underbrace{
\begin{bmatrix}
M + \Delta t D_a &  & & \\
- M & M + \Delta t D_a &  & \\
& - M & M + \Delta t D_a &  & \\
&  & \ddots & \ddots & \\
&  &  & -M & M + \Delta t D_a \\
\end{bmatrix}
}_{\Cal K}
\begin{bmatrix}
y_1\\
y_2\\
y_3 \\
\vdots \\
y_{N_t}
\end{bmatrix}
$$
$$
\qquad \qquad \qquad
\qquad \qquad \qquad
- \Delta t
\underbrace{
\begin{bmatrix}
D_c& 0 & \cdots& \\
& D_c   &  & \\
&  & D_c   &  & \\
&  &  & \ddots & \\
&  &  &  & D_c   \\
\end{bmatrix}
}_{\Cal C}
\begin{bmatrix}
u_1\\
u_2\\
u_3 \\
\vdots \\
u_{N_t}
\end{bmatrix}
\hspace{-1mm}= \hspace{-1mm}
\begin{bmatrix}
My_0 + \Delta t g_1\\
0 + \Delta t  g_2\\
0 + \Delta t  g_3\\
\vdots \\
0 + \Delta t  g_{N_t}
\end{bmatrix}.
\]
\no Then, the above system could be written in the following form:
\[
\Cal K \pmb y - \Delta t \Cal C \pmb u = \fb{\Cal M \pmb y_0 + \Delta t  \pmb g},
\]
\fb{\A{where $\Cal M$ is a block diagonal matrix} in $ \mathbb R^{\Cal N  \cdot N_t} \times \mathbb R^{\Cal N  \cdot  N_t} $ which \A{diagonal} entries are $[M, \cdots, M]$}.
In the same way we can analyze the adjoint equation: we have to solve the first equation of the optimality system \eqref{strong_form_optimality_system} at each time step as follows:
\[
Mp_{k}  = M p_{k+1}  + \Delta t ( - M y_{k} - D_a^T p_{k} + My_{d_{k}} )
\spazio \B{\text{for } k \in \{ N_t - 1, N_t - 2, \dots, 1 \}}.
\]

\no As we did for the state equation, one could use an all-at-once strategy and consider the following system:
\[
\underbrace{
\begin{bmatrix}
M + \Delta t D_a^T & -M  & & \\
& M + \Delta t D_a^T & -M  & \\
&  & \ddots & \ddots  & \\
&  &  & M + \Delta t D_a^T & -M  \\
&  &  &  & M + \Delta t D_a^T \\
\end{bmatrix}
}_{\Cal K^T}
\begin{bmatrix}
p_1\\
p_2\\
p_3 \\
\vdots \\
p_{N_t}
\end{bmatrix}
+
 \begin{bmatrix}
\Delta t M y_1\\\
\Delta t M y_2\\
\Delta t M y_3 \\
\vdots \\
\Delta t M  y_{N_t}
\end{bmatrix}
=
\begin{bmatrix}
\Delta t My_{d_1} \\
\Delta t My_{d_2} \\
\Delta t My_{d_3} \\
\vdots \\
\Delta t My_{d_{N_t}}
\end{bmatrix}.
\]

\no Then, the adjoint system to be solved is:
\[
\Cal K^T \pmb p + \Delta t \Cal M \pmb y = \Delta t \Cal M \pmb y_d.
\]

\no Now, if we consider the optimality equation given by the differentiation of \eqref{functional} with respect to the control variable, at every time step we have to solve the equation
$\alpha \Delta t M u_k - \Delta t D_c^T p_k = 0.$ In a vector notation we have $\alpha \Delta t \Cal M \pmb u - \Delta t \Cal C^T \pmb p = 0$.
In the end, the final system considered and  solved through an one shot approach is the following:
\begin{equation}
\label{one_shot_system}
\begin{bmatrix}
\Delta t  {\Cal M} & 0 & \Cal {K}^T  \\
0 & \alpha \Delta t \Cal M &  -\Delta t \Cal C^T \\
\Cal K & -\Delta t \Cal C & 0 \\
\end{bmatrix}
\begin{bmatrix}
\pmb y \\
\pmb u \\
\pmb p \\
\end{bmatrix}
=
\begin{bmatrix}
\Delta t {\Cal {M}} \pmb y_d \\
0 \\
\fb{ {\Cal {M}} \pmb y_0 + \Delta t \pmb g} \\
\end{bmatrix}.
\end{equation}
Now, let us denote with
\begin{align*}
A =
\begin{bmatrix}
\Delta t {\Cal M} & 0 \\
0 & \alpha \Delta t \Cal M \\
\end{bmatrix},
\quad  B =
\begin{bmatrix}
\Cal K & -\Delta t \Cal C \\
\end{bmatrix},
\quad F =
\begin{bmatrix}
\Delta t {\Cal {M}} \pmb y_d  \\
0 \\
\end{bmatrix}
\quad \text{and } \quad
G = \fb{{\Cal M}\pmb y_0 + \Delta t  \pmb g}.
\end{align*}
Then, the system \eqref{one_shot_system} can be written as follows:
\begin{equation}
\label{one_shot_system_saddle_point}
\begin{bmatrix}
A & B^T \\
B & 0 \\
\end{bmatrix}
\begin{bmatrix}
\pmb x \\
\pmb p
\end{bmatrix}
=
\begin{bmatrix}
F \\
G
\end{bmatrix}.
\end{equation}
The global dimension of the full order  model presented is $\Cal N_{\text{tot}} = 3 \times N_t \times \Cal N$. \B{We underline that the saddle point structure does not depend on the used time approximation scheme: indeed, this framework can be extended to more general time discretization. }We managed to solve the linear system \eqref{one_shot_system} through a direct approach. Nevertheless, iterative methods based on Krylov solvers and Schur preconditioning are a very common choice in the solution of saddle point structures \cite{benzi_golub_liesen_2005} and they have been applied in the optimal control framework in \cite{schoberl2007symmetric,Stoll}, for example.\\
Now that we have introduced the FE discretization as our full order approximation, we can exploit it in order to build a reduced system, as it will be specified in the following section.
\section{POD-Galerkin ROM Applied to Parabolic Time Dependent \ocp s}
\label{sec_ROM} 
This section aims at introducing ROM approximation for time dependent \ocp s into the framework of saddle point problems by applying the techniques already used for the steady case in  \cite{negri2015reduced,negri2013reduced}. After a general introduction of the main ideas behind ROM applicability, we will briefly introduce the affine assumption over the bilinear and linear forms involved in the problem formulation. Then, in \A{Section} \ref{PODsec} we will describe the POD-Galerkin algorithm, see for example \cite{ballarin2015supremizer,burkardt2006pod,Chapelle2013,hesthaven2015certified} as references, and we will extend it to parametrized time dependent \ocp s. In the end, we will show how to guarantee the well-posedness of reduced  \ocp s thanks to the aggregated space strategy \cite{dede2010reduced,negri2015reduced,negri2013reduced}.

\subsection{Reduced Problem Formulation}
In \A{Section} \ref{problem}, we showed that linear quadratic time dependent \ocp s could be seen as saddle point problems of the form \eqref{linear_optimality system}. We recall to the reader that our state-control variable is $x(\boldsymbol{\mu}) = (y(\boldsymbol{\mu}),u(\boldsymbol{\mu})) \in \Cal X$: in this case we \A{highlight} the parameter dependency \A{which is of crucial importance in the understanding of the reduced formulation concepts}.
We assume that our parametric solution defines a smooth \textit{solution manifold}
$$
\mathscr M = \{ (x(\boldsymbol{\mu}), p(\boldsymbol{\mu}))\;| \; \boldsymbol{\mu} \in \mathscr P\}.
$$
The FE approximation is reflected also in the solution manifold: indeed one can define a approximated solution manifold:
$$
\mathscr M \disc = \{ (x\disc(\boldsymbol{\mu}), p\disc(\boldsymbol{\mu}))\;| \; \boldsymbol{\mu} \in \mathscr P\}.
$$
\no Once again, we assume that also the approximated manifold has a smooth dependence with respect to $\boldsymbol{\mu}$. Reduced order methodology wants to recover the structure of $\mathscr M \disc$ through basis \A{functions} derived from properly chosen full order solutions $x \disc (\boldsymbol{\mu})$ and $ p \disc(\boldsymbol{\mu})$ called snapshots. In other words, we construct the reduced \A{bases exploiting FE solutions evaluated for some specific $\bmu \in \mathscr P$}. Let us assume  to \A{have already built} $\Cal X_N \subset \Cal X \disc \subset \Cal X$ and $\Cal Y_t{_N} \subset \Cal Y_{t} \disc \subset \Cal Y_t$ as reduced product state-control space and reduced adjoint space, respectively. We have all the ingredients to define the reduced problem as follows: given $\boldsymbol{\mu} \in \mathscr P$, find $(x_N(\boldsymbol{\mu}), p_N(\boldsymbol{\mu})) \in \Cal X_N \times \Cal Y_t{_N}$ such that
\begin{equation}
\label{RB_OCP}
\begin{cases}
\Cal A(x_N(\boldsymbol{\mu}),w_N) + \Cal B(w_N, p_N(\boldsymbol{\mu}); \boldsymbol{\mu}) = 
\displaystyle
\intTime{
\la F(\boldsymbol{\mu}), w_N \ra} & \forall w_N \in \Cal {X}_N, \\
\Cal B(x_N(\boldsymbol{\mu}), q_N; \boldsymbol{\mu}) =  
\displaystyle
\intTime{
\la G(\boldsymbol{\mu}), q_N \ra} & \forall q_N \in \Cal Y_t{_N}.
\end{cases}
\end{equation}
Also at the reduced level, in order to assert the well-posedness of the reduced saddle point problem \eqref{RB_OCP}, Brezzi's theorem has to be verified. In \A{Section} \ref{agg}, we will show how to recover the existence and the uniqueness of the reduced minimizing solution $(x_N(\boldsymbol{\mu}), p_N(\boldsymbol{\mu})) \in \Cal X_N \times \Cal Y_t{_N}$.

\subsection{\A{Affine} Assumption: \fb{Offline--Online} decomposition}
\label{aff}
Let us briefly underline the crucial hypothesis which guarantees efficient applicability of reduced order methods: \A{affine} assumption. A problem in \A{a} saddle point framework \eqref{linear_optimality system} is \A{affinly decomposed with respect to the parameter} if the involved bilinear forms and functionals can be recast as:
\begin{equation}
\begin{matrix}
& \qquad \Cal A(x, w; \boldsymbol{\mu}) =\displaystyle  \sum_{q=1}^{Q_\Cal{A}} \Theta_\Cal{A}^q(\boldsymbol{\mu})\Cal{A}^q(x,w), & 
& \qquad \Cal B(w,p; \boldsymbol{\mu}) =\displaystyle  \sum_{q=1}^{Q_\Cal{B}} \Theta_\Cal{B}^q(\boldsymbol{\mu})\Cal{B}^q(w,p), \\
& \qquad  \la G(\boldsymbol{\mu}), p \ra =\displaystyle  \sum_{q=1}^{Q_G} \Theta_G^q(\boldsymbol{\mu})\la G^q , p \ra, & 
& \qquad  \la F(\boldsymbol{\mu}), w \ra =\displaystyle  \sum_{q=1}^{Q_F} \Theta_F^q(\boldsymbol{\mu})\la F^q , w \ra,
\end{matrix}
\end{equation}
\no  for some finite $Q_\Cal{A}, Q_\Cal{B}, Q_G, Q_F$, where $\Theta_\Cal{A}^q,\Theta_\Cal{B}^q, \Theta_G^q, \Theta_F^q$ are $\boldsymbol{\mu}-$dependent smooth functions, whereas $\Cal A^q,\Cal B^q$, $G^q, F^q$ are $\boldsymbol{\mu} -$independent bilinear forms and functionals.
\\Thanks to this assumption, the solving process of our \ocp $\:$ can be divided in two different phases: an \textbf{offline} stage where the reduced spaces are derived and the $\boldsymbol{\mu}-$independent quantities are assembled. This stage could be very expensive but it is performed only once. Then, an \textbf{online} stage follows and all the $\boldsymbol{\mu}-$dependent quantities are assembled and the whole reduced system is solved. The online phase is performed for every new parameter evaluation in order to study different physical and/or geometrical configurations.

\no If the problem does not fulfill the \A{affine} assumption, \A{the} empirical interpolation method can be exploited in order to recover it, as presented in \cite{barrault2004empirical} or in \cite[Chapter 5]{hesthaven2015certified}.
\subsection{POD Algorithm for OCP($\boldsymbol{\mu}$)s}

\label{PODsec}
In this section we introduce the POD-Galerkin strategy which we exploited in order to build the reduced spaces needed for time dependent \ocp s, following \cite{ballarin2015supremizer,burkardt2006pod,Chapelle2013,hesthaven2015certified}.\\
The first step of the the POD-Galerkin approach is to choose a discrete subset of parameters $\mathscr P_h \subset \mathscr P$. Thanks to this new finite dimensional parametric set, we can define a specific solution manifold:
$$
\mathscr M \disc (\mathscr P_h) = \{ ( x \disc(\boldsymbol{\mu}), p\disc(\boldsymbol{\mu})) \; | \; \boldsymbol{\mu} \in \mathscr P_h\},
$$
which cardinality is $N_{\text{max}} = |\mathscr P_h|$ and which satisfies the inclusion
$
\mathscr M \disc (\mathscr P_h) \subset \mathscr M\disc
$
since $\mathscr P_h \subset \mathscr P$. When the finite parameter space $\mathscr P_h$ is large enough, the manifold $\mathscr M \disc (\mathscr P_h) $ can be a reliable representation of the discrete manifold $\mathscr M \disc$.
The POD-Galerkin approach \A{compresses the redundant information contained in the snapshots of $\mathscr M \disc (\mathscr P_h)$}.
\\ We exploited a \tit{partitioned approach}, i.e. the POD algorithm has been applied separately for state, control and adjoint variables.
\no In the end, the POD algorithm provides $N-$dimensional reduced spaces that minimize the quantities:
\begin{equation}
\label{crit} \hspace{-1.5cm}
\sqrt{\frac{1}{N_{max}}
\sum_{\boldsymbol{\mu} \in \mathscr P_h} \underset{z_N \in {{\Cal {Y}}_t}_{N}}{\text{min }} \norm{y\disc(\boldsymbol{\mu}) - z_N}_{\Cal Y_t}^2},
\hspace{.2cm}
\sqrt{\frac{1}{N_{max}}
\sum_{\boldsymbol{\mu} \in \mathscr P_h} \underset{v_N \in \Cal U_N}{\text{min}} \norm{u\disc(\boldsymbol{\mu}) - v_N}_{\Cal U}^2},
\hspace{.2cm}
\sqrt{\frac{1}{N_{max}}
\sum_{\boldsymbol{\mu} \in \mathscr P_h} \underset{q_N \in {\Cal {Y}_t}_{N}}{\text{min}} \norm{p\disc(\boldsymbol{\mu}) - q_N}^2_{\Cal Y_t}}.
\end{equation} 
\no We are going to introduce the POD-Galerkin procedure for the state solution $y(\bmu)$. The same strategy has been used for \fb{control} and adjoint variables \fb{as well}. 
\\ Let us consider ordered parameters $\boldsymbol{\mu}_1, \dots, \boldsymbol{\mu}_{N_{max}}\in \mathscr P_h$ and the resulting ordered FE solutions \\$y\disc(\boldsymbol{\mu}_1), \dots, y\disc(\boldsymbol{\mu}_{N_{max}})$. Furthermore, we define the correlation matrix $\mbf C^y \in \mbb R^{N_{max} \times N_{max}}$ of snapshots of the state variable, i.e.:
$$
\mbf C_{ml}^y = \frac{1}{N_{max}}(y\disc(\boldsymbol{\mu}_m),y\disc(\boldsymbol{\mu}_l))_{\Cal {Y}_t}, \hspace{1cm} 1 \leq m,l \leq N_{max}.
$$
\no We look for the $N$-largest eigenvalue-eigenvector pairs $(\lambda_n^y, v_n^y)$, which solve the following equations:
$$
\mbf C^y v_n^y = \lambda_n^y v_n^y, \hspace{1cm} 1 \leq n \leq N, 
$$ 
\no with $\norm {v_n^y} = 1$. 
Let us order the eigenvalues $\lambda_1^{\fb{y}} \geq \lambda_2^{\fb{y}} \geq \cdots
\geq \lambda_N^{\fb{y}}$ from the \fb{largest} to the smallest. This order reflects on the basis functions $\{\xi_1^{\fb{y}}, \dots, \xi_N^{\fb{y}}\}$ of the reduced space
$\Cal Y_t{_N}= \text{span }\{\xi_1^y, \dots, \xi_N^y\}$. The basis satisfies:
$$
\xi_n^y = \displaystyle \frac{1}{\sqrt{\fb{\lambda_m^y}}}\sum_{m = 1}^M (v_n^y)_m y\disc(\boldsymbol{\mu}_m), \hspace{1cm} 1 \leq n \leq N,
$$
\no where $(v_n^y)_m$ is \emph{m-th} component of the state eigenvector $v_n^y \in \mbb R^M$.

\begin{remark} 
We underline that, \fb{even if} we applied the \fb{POD algorithm} separately for the different variables, \fb{we have not separated time instances, i.e. each snapshot still contains the solution at \emph{all} temporal steps}. \fb{In this way, the reduced basis functions comply with the space-time formulation introduced in \eqref{FEOCP}, and the resulting POD-Galerkin ROM is a space-time reduced order model.}
\end{remark}

\subsection{Aggregated Spaces Approach}

\label{agg}
As we underlined in \A{Section} \ref{problem}, the adjoint variable $p(\bmu)$ is considered in the state space $\mathcal Y_t$, \fb{in order} to ensure the well-posedness of the whole \ocp . It is well known that a POD for state and adjoint variables will not lead necessarily to the same reduced space approximation. Indeed, let us assume to have applied the POD algorithm \A{with the same value of $N_{\text{max}}$ and retaining the first $N$ eigenvalues for all the involved variables}, as described in \A{Section} \ref{PODsec}: the procedure provides reduced spaces for state, control and adjoint variables as
\begin{align*}
& \Cal Y_{t_{N}} = \text{span}\{y\disc (\boldsymbol{\mu}^n) \; n = 1, \dots, N\}, \\
& \Cal U_{N} = \text{span}\{u\disc (\boldsymbol{\mu}^n) \; n = 1, \dots, N\}, \\
& \Cal Q_{t_{N}} = \text{span}\{p\disc (\boldsymbol{\mu}^n) \; n = 1, \dots, N\},
\end{align*}
respectively. As already \A{done} for the continuous and full order versions of the problem, we define the product space $\Cal X_N = \Cal Y_{t_{N}} \times \Cal U_{N}$. Once the reduced spaces are available, it remains to prove if the Brezzi's theorem is still valid, i.e. if the reduced saddle point problem \eqref{RB_OCP} admits an unique solution. The continuity of the bilinear forms $\Cal A \cd$ and $\Cal B (\cdot, \cdot; \pmb \mu)$ are directly inherited from the FE approximation, as well as the coercivity of $\Cal A \cd$ over  the kernel
$\Cal X_0^N = \{ w_N \in \mathcal X_N : \Cal B(w_N, q; \pmb \mu) = 0, \; \forall q \in \Cal Q_{t_{N}} \}$, even if the state and the adjoint reduced spaces do not coincide. Although, it is not guaranteed  the fulfillment of the \emph{reduced inf-sup condition}. Indeed, only the following inequality holds: 
\begin{equation}
\label{reduced_infsup}
\inf_{0 \neq q \in \Cal Y_{t_{N}}} \sup_{0 \neq x \in \Cal X_N } \frac{\Cal B(x, q; \bmu)}{\norm{x}_{\Cal X}\norm{q}_{\Cal Y_t}} \geq  \beta_N(\bmu) > 0,
\end{equation}
i.e. the \emph{reduced inf-sup} \emph{condition} \eqref{reduced_infsup} in verified only when $\Cal Y_{t_{N}} \equiv \Cal Q_{t_{N}}$, which is not the case for a standard POD approach. \\
In order to avoid this \fb{inconvenience}, we exploit the aggregated spaces technique as presented in \cite{dede2010reduced,negri2015reduced,negri2013reduced}.  The main feature of this approach is to define a common space for state and adjoint variables, given by
$$
Z_N = \text{span }\{y\disc (\boldsymbol{\mu}^n), p\disc(\boldsymbol{\mu}^n), \; n = 1, \dots, N\}. 
$$
The space $Z_N$ is then used \fb{to} describe both the reduced state variable $y_N(\bmu)$ and the reduced adjoint variable \fb{$p_N(\bmu)$}. The new product space is now $\Cal X_N = Z_N \times \Cal U_N$, where the control space is of the standard form
$$
\Cal U_N = \text{span }\{ u \disc (\boldsymbol{\mu} ^n), \; n = 1, \dots, N\}.
$$
This choice will lead to a global dimension $N_{\text{tot}} =5N$. Thanks to this strategy, the reduced optimality system is well-posed since all the hypotheses of Brezzi's theorem hold with the \emph{reduced inf-sup} of the following form
\begin{equation}
\label{reduced_infsup_Z}
\inf_{0 \neq q \in {Z_{N}}} \sup_{0 \neq x \in \Cal X_N } \frac{\Cal B(x, q; \bmu)}{\norm{x}_{\Cal X}\norm{q}_{\Cal Y_t}} \geq  \beta_N(\bmu) > 0.
\end{equation}
\no Now we have all the necessary notions needed in order to show some applications of ROM for time dependent \ocp s. In the next \A{Section} we will show how advantageous reduced modelling could be in this very costly context.

\section{Numerical Results: Time Dependent \ocp $\:$ for Graetz flows}
In this section we are going to present a numerical example in order to validate the performances of POD-Galerkin method for time dependent \ocp s: we will apply our methodology to a time dependent version of the test case proposed in \cite{negri2013reduced}.

\label{ADR_OCP}
\no \A{The proposed test case} deals with a time dependent \ocp $\:$ governed by a Graetz flow with a control over \A{the boundary $\Gamma_C = ([1, 1+\mu_3] \times \{0\}) \cup ([1, 1+\mu_3] \times \{1\})$}, which is represented in \A{Figure} \ref{domainADR}. In this case, $\mu_3$ is a geometrical parameter which stretches the length of \both{$\Omega_2(\mu_3) := [1, 1+\mu_3]\times [0.2, 0.8]$ }  and \both{$\Omega_3(\mu_3) :=[1, 1+\mu_3]\times [0, 0.2] \cup [1, 1+\mu_3]\times [0.8,1]$ } as \A{Figure} \ref{domainADR} shows.
\no We now introduce the other two parameters: $\mu_1$ which represents the diffusivity coefficient of the system and $\mu_2$ which is the desired profile solution we want to reach in \A{$\Omega_3(\mu_3)$}.
The parameter $\pmb \mu = [\mu_1, \mu_2, \mu_3]$ is considered in $\mathscr P = [1/20,1/6] \times [1,3] \times [1/2, 3]$. 


\no  Let us specify the function spaces $Y$ and $U$ needed. For this test case we define $y \in \Cal Y_t$ where $Y = H^1_{\Gamma_{D}}(\Omega(\mu_3))$ and $u \in \Cal U$ where $U=L^2(\Gamma_C(\mu_3))$. Furthermore, let $\Cal X$ be the product space $\Cal Y_t \times \Cal U$.
All the data of the \A{test case} are recap in \A{Table} \ref{table_data_ADR}.
The problem we consider reads: given $\pmb \mu \in \mathscr P$, find the state-control variable 
$(y, u) \in \Cal X$ which solves:
\begin{equation}
\label{output}
\min_{(y,u) \in \Cal X} J(y,u; \pmb \mu ) = \min_{(y,u) \in \Cal X} 
\half \int_0^T \int_{\Omega_{3}}(y - y_d(\pmb \mu))^2dxdt + 
\alf \int_0^T \int_{\Gamma_C}u^2dxdt
\end{equation}
constrained to the equation
\begin{equation}
\begin{cases}
\displaystyle \dt{y} + \mu_1 \Delta y + x_{{2}}(1 - x_{{2}})\frac{\partial y}{\partial x_{{1}}} = 0 & \text{in } \Omega(\mu_3) \times \B{(0,T)}, \\
y = 1 & \text{on } \Gamma_D(\mu_3) \times \B{(0,T)}, \\
\displaystyle \mu_1 \dn{y} = u & \text{on } \Gamma_C(\mu_3) \times \B{(0,T)}, \vspace{1mm}\\
\displaystyle \mu_1 \dn{y} = 0 & \text{on } \Gamma_N(\mu_3) \times \B{(0,T)}, \\
y = y_0 &  \text{in } \Omega(\mu_3) \times \{0\},
\end{cases}
\end{equation}
where $x_1$ and $x_2$ are the spatial components, $y_0$ is the null function in the domain which respects the boundary conditions and $y_d(\bmu) \equiv \mu_2$.
As already presented in \A{Section} \ref{problem}, we applied \A{a} Lagrangian approach and the optimize-then-discretize technique in order to recover the following optimality system: given $\pmb \mu \in \mathscr P$, find $((y, u), p) \in 
\Cal X \times \Cal Y_t$ such that
\begin{equation}
\begin{cases}
 y - \displaystyle \dt{p} + \mu_1 \Delta p - x_{{2}}(1 - x_{{2}})\frac{\partial p}{\partial x_{{1}}} = y_d(\bmu)& \text{in } \Omega(\mu_3) \times \B{(0,T)}, \\
p = 0 & \text{on } \Gamma_D(\mu_3) \times \B{(0,T)} \\
\displaystyle \mu_1 \dn{p} = 0 & \text{on } \Gamma_N(\mu_3) \times \B{(0,T)}, \\
p = 0 & \text{in } \Omega(\mu_3) \times \{T\}, \\
\alpha u = p & \text{in } \Gamma_C(\mu_3) \times \B{(0,T)}, \\
\displaystyle \dt{y} + \mu_1 \Delta y + x_{{2}}(1 - x_{{2}})\frac{\partial y}{\partial x_{{1}}} = 0 & \text{in } \Omega(\mu_3) \times \B{(0,T)}, \\
y = 1 & \text{on } \Gamma_D \times \B{(0,T)}, \\
\displaystyle \mu_1 \dn{ y} = u & \text{on } \Gamma_C(\mu_3) \times \B{(0,T)}, \vspace{1mm} \\
\displaystyle \mu_1 \dn{ y} = 0 & \text{on } \Gamma_N(\mu_3) \times \B{(0,T)}, \\
y = y_0 & \text{in } \Omega(\mu_3) \times \{0\}.
\end{cases}
\end{equation}

\no The problem has been solved exploiting the following strategy: first of all, we traced back the original problem into the reference domain presented in \A{Figure} \ref{domainADR}. \A{The reference domain corresponds to $\mu_3 = 1$}. The full order discretization is performed as described in \A{Section} \ref{Algebraic_parabolic}: we used $\Delta t =1/6$
over the interval $\B{(0,T)} = \B{(0,5)}$, with a resulting number of time steps $N_t = 30$. Moreover, for the space discretization, we used $\mathbb P^1$ elements for all the variables involved \B{with $\Cal N = 3487$}, working with function spaces of the form $\Cal X\disc := {\Cal {Y}_t} \disc \times \Cal U \disc$ and ${\Cal Y_t} \disc$, for every time step of state-control and adjoint variables, respectively. The total number of degree of freedom of the full order approximation is $\Cal N_{\text{tot}} = 3\times N_t \times {\Cal {N}} = 313'830$. We built the reduced spaces applying the POD-Galerkin approach as presented in \A{Section} \ref{PODsec}. For all the variables we choose $N_{\text{max}} = 70$ snapshots. The basis functions  were obtained retaining the first $N = 35$ eigenvectors of the snapshots correlation matrix: from now on we will define $N$ as the \emph{basis number}. In order to guarantee the well-posedness of the reduced saddle point problem arising from the constrained optimization, we exploited aggregated space technique described in \A{Section} \ref{agg}: it led to a total reduced dimension of $N_{\text{tot}}= 5N = 175$.  The basis considered are sufficient and they well represent the optimal solution as one can observe from the average relative error\footnote{The error for state, control and adjoint variables are presented in the following norms: $\norm{y\disc - y_N}^2_{H^1}$, \\$\norm{u\disc - u_N}^2_{L^2}$ and $\norm{p\disc - p_N}^2_{H^1}$, respectively.} plots in \A{Figure} \ref{error_ADR} and the solution plots in \A{Figure} \ref{simulation_ADR}.  The \A{FE} simulations and the reduced simulations are compared in \A{Figure} \ref{simulation_ADR}: they match for different time instances. In \A{Figure} \ref{error_ADR}, we show the average relative errors over a testing set of $50$ parameters, \B{uniformly distributed}: as expected, it decreases for a high basis number $N$, with a minimum value of $10^{-4}$ for all the variables. In \A{Figure} \ref{error_ADR}, we also present the average relative error between the FE and the ROM values of the functional $J(\cdot, \cdot; \pmb \mu)$, i.e. our \emph{output}.  We specify that the effort needed for solving an offline phase for time dependent \ocp s drastically increases for high values of $N_{\text{max}}$: still, our choice for the number of snapshots gave us a good overview of the whole parametrized system in a reasonable amount of time. Let us now focus on the computational time involved both in the full and in the reduced order simulations.
\no In \A{Figure} \ref{speedup_ADR}, the \emph{speedup} index for this \A{test case} is presented. The speedup represents how many reduced order simulations can be performed in the time of a single full order FE element simulation. It reaches a maximum value of $1,8\cdot 10^5$, while the lowest values associated to an increasing value of $N$ are not below $10^5$. The speedup index underlines how convenient is ROM system for repeated parametric instances of time dependent \ocp s, since the very expensive formulation of the whole system \eqref{one_shot_system} is projected in a low dimensional framework which recovers the evolution of the optimality system and permits to study several configurations in the online phase.  
\begin{table}[H]
\centering
\caption{Data for the \ocp  $\:$ governed by a Graetz flow.}
\label{table_data_ADR}
\begin{tabular}{ c | c }
\toprule
\textbf{Data} & \textbf{Values} \\
\midrule
$\mathscr P$ & $[1/20,1/6] \times [1,3] \times [1/2, 3]$\\
\midrule
 $(\mu_1, \mu_2, \mu_3,  \alpha)$ &  $(1/12, 2, 5/2, 10^{-2})$ \\
\midrule
\A{$N_{\text{max}}$} & 70 \\ 
\midrule
Basis Number $N$ & 35 \\
\midrule
Sampling Distribution & Uniform \\
\midrule
$\Cal N_{\text{tot}}$ & $313'830$ \\
\midrule
\A{$N_{t}$}& \A{$30$} \\
\midrule
ROM System Dimension & $175$ \\
\bottomrule
\end{tabular}
\end{table}
\begin{figure}[H]
\begin{center}
\includegraphics[scale = .33]{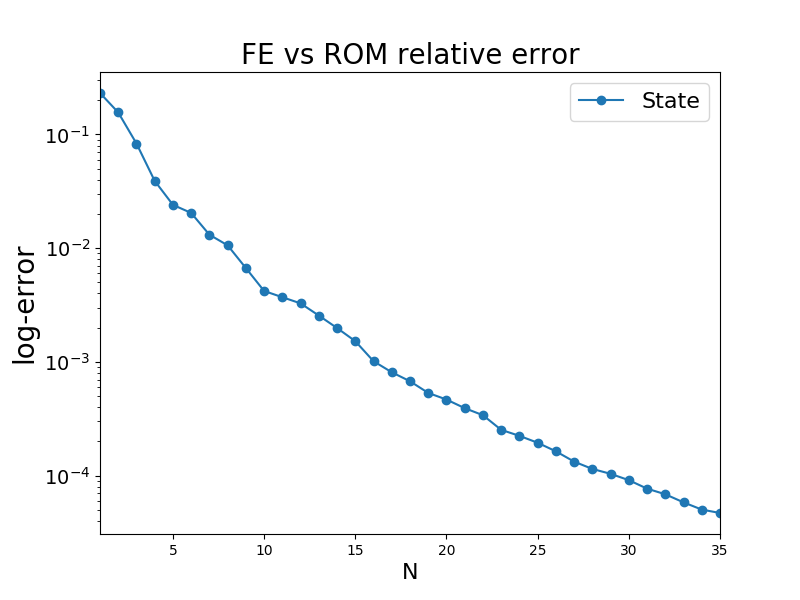}\includegraphics[scale = .33]{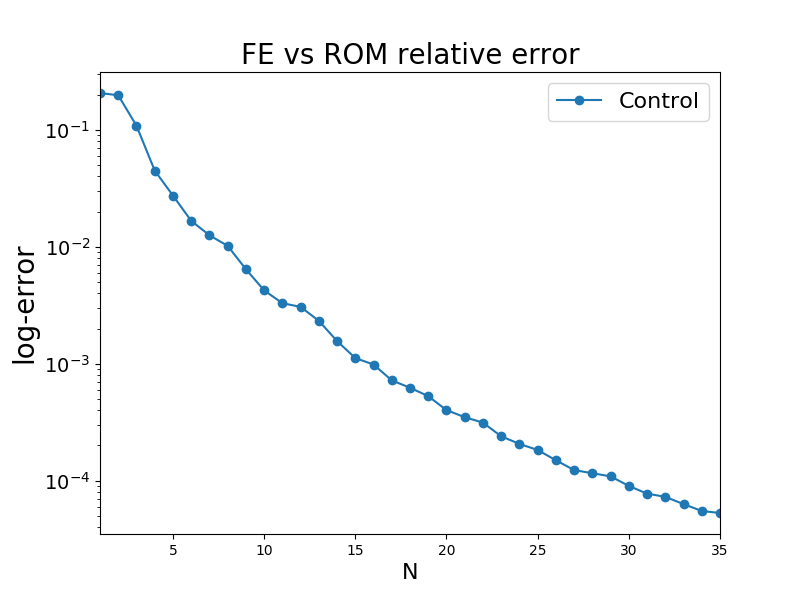} \vspace{-2mm}\\
\includegraphics[scale = .33]{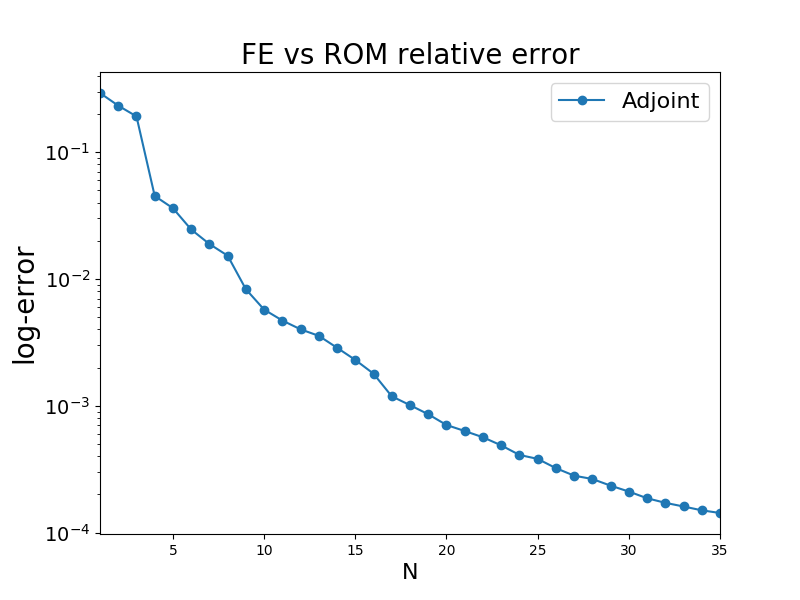}\includegraphics[scale = .33]{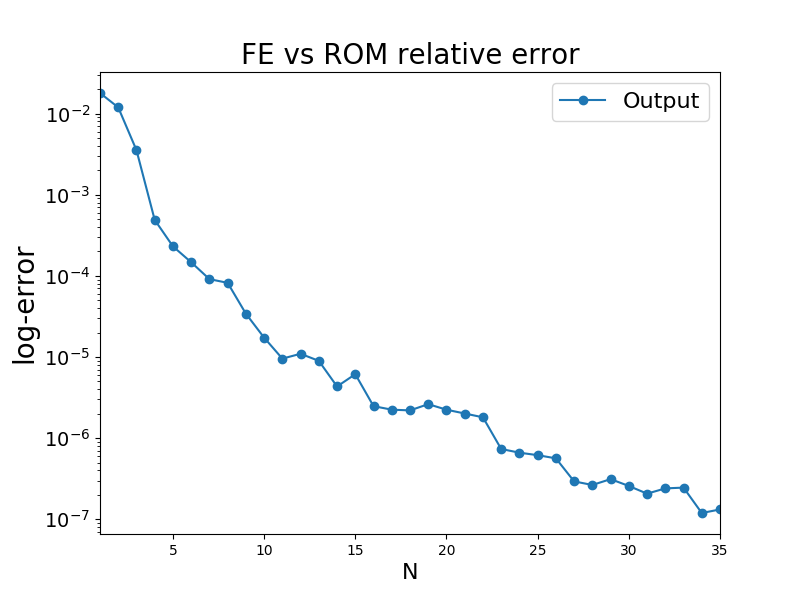}
\caption{FE vs ROM relative errors with respect to the basis number $N$ in logarithmic scale. Top left: state error; top right: control error; bottom left: adjoint error; bottom right: output error.}
\label{error_ADR}
\end{center}
\end{figure}
\vspace{-1.7cm}

\begin{figure}[H]
\hspace{-1cm}\includegraphics[scale = .2]{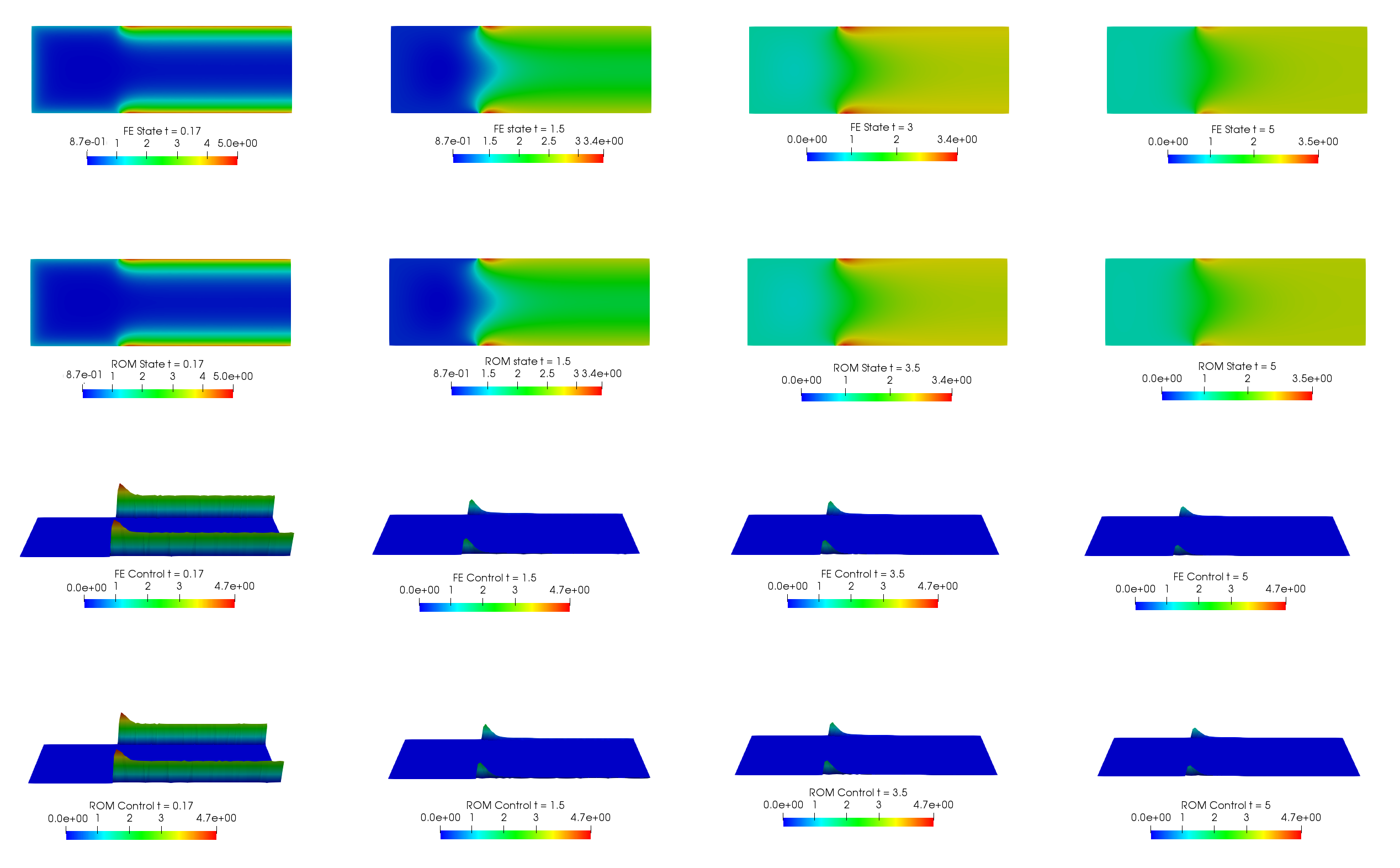}
\caption{FE vs ROM simulations of state and control variables for $\pmb \mu = [1/12, 2, 2.5]$ and $\alpha = 10^{-2}$. First row: FE state for $t=0.17,1.5,3.5, 5$; second row: ROM state for $t=0.17,1.5,3.5, 5$; third row: FE control for $t=0.17,1.5,3.5, 5$; fourth row: ROM control for $t=0.17,1.5,3.5, 5$. In order to better visualize the boundary control, we represent it in a third dimension.}
\label{simulation_ADR}
\end{figure}

\begin{figure}[H]
\begin{center}
\includegraphics[scale = .33]{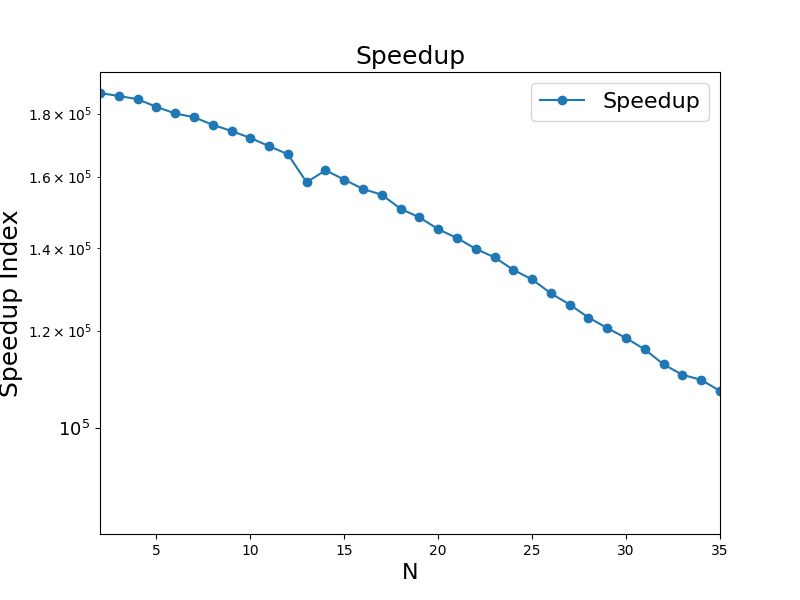}
\caption{Speedup analysis in logarithmic scale shown with respect to the basis number $N$.}
\label{speedup_ADR}
\end{center}
\end{figure}

\begin{figure}[H]
\begin{center}
\includegraphics[scale = .35]{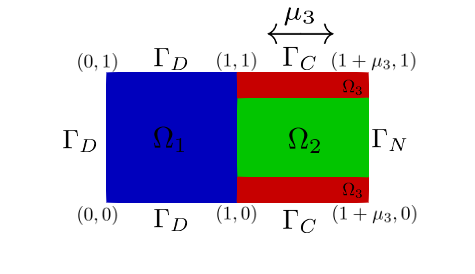}
\caption{\both{Considered reference } domain and subdomains \both{for $\mu_3 = 1$}.}
\label{domainADR}
\end{center}
\end{figure}

\section{Linear Quadratic Time Dependent \ocp s governed by Stokes Equations}
\label{Stokes_OCP}
In this section, we introduce a distributed \ocp $\:$ governed by time dependent Stokes equations. As we did for the parabolic case in section \ref{problem}, first of all, we provide the proof of well-posedness of this specific linear quadratic \ocp $\:$ in a saddle point framework. Then, we will adapt the space-time discretization and the aggregated POD-Galerkin technique of \A{Sections} \ref{Algebraic_parabolic} and \ref{agg}, respectively, to this more complex problem.

\subsection{Problem Formulation}
First of all, we would like to specify the nature of the parametrized problem: in our applications we consider $\pmb \mu = [\mu_1, \mu_2]$. The \A{parameters} $\mu_1$ and $\mu_2$ characterize the physics and the geometry of the governing equation, respectively.
 Let us consider the parametrized domain $\Omega (\mu_2) \subset \mathbb R^2$. We indicate with 
$\Gamma_D(\mu_2)$
the portion of the domain where Dirichlet boundary conditions are applied, while $\Gamma_{N}(\mu_2)$ is the part of $\partial \Omega$ characterized by Neumann conditions.  Let us specify the function spaces involved in this example:
let us introduce the space $V = H^1_{\Gamma_D(\mu_2)}(\Omega(\mu_2))$, the space $P = L^2(\Omega(\mu_2))$,
$\Cal V_t = \{ y \in \Cal V := L^2(0,T; V) \; \text{s.t.} \; \displaystyle \dt{y} \in L^2(0,T; V\dual) \}$ and
$\Cal P = L^2(0,T; P)$.
Our parameter dependent state variable $(y, p) \in \Cal Y_t := \Cal V_t \times \Cal P$ is controlled by the variable 
$u \in \Cal U := L^2(0,T; L^2(\Omega(\mu_2)))$. Following the structure described in \A{Section} \ref{problem}, we will exploit the velocity-pressure-control variable 
$x \in \Cal X:= \Cal Y_t \times \Cal U$, where $x := x(\bmu)$ represents the triplet 
$(y, p, u) := (y(\bmu), p(\bmu), u(\bmu))$. The norm of the space $\Cal X$ is $\norm{x}_{\Cal X}^2 = \norm{y}_{\Cal V_t}^2 + \norm{p}_{\Cal P}^2 + \norm{u}_{\Cal U}^2$.
In order to build the \ocp $\;$ we need to define the adjoint variable 
$(\lambda, \xi) \in \Cal Y_t$. From now on, the \fb{pair} $(\lambda, \xi)$ will be indicated with $\sigma$.
For a given $\bmu$ in a suitable parameter space  $\mathscr P \subset \mathbb R^2$, we want to find $(y,p,u) \in \Cal X $ which solves:
\begin{equation}
\label{func_Stokes}
\min_{(y,p,u) \in \Cal X} \half \intTimeSpace{(y - y_d)^2}  + \alf \intTimeSpace{u^2}\\
\end{equation}
constrained to the time dependent Stokes equations:
\begin{equation}
\label{P_Stokes}
\begin{cases}
\displaystyle
\dt{y} - \mu_1 \Delta y + \nabla p = u & \text{\spazio in } \Omega(\mu_2) \times \B{(0,T)} , \\
\dive(y) = 0 & \text{\spazio in } \Omega(\mu_2) \times \B{(0,T)} ,  \\
\displaystyle \frac{\partial y}{ \partial n} = 0 & \text{\spazio on  }  \Gamma_N(\mu_2) \times \B{(0,T)} , \\
y = g & \text{\spazio on  }  \Gamma_D(\mu_2) \times \B{(0,T)} , \\
y(0) = y_0 & \text{\spazio in } \Omega(\mu_2) \times \{\fb{0}\},
\end{cases}
\end{equation}
where $y_d \in \Cal V_t$ is a desired velocity profile \A{defined on the whole domain}. Also in this application, 
$\norm{y_0}_{L^2(\Omega(\mu_2))} = 0$. Again, for the sake of notation, we have omitted the parameter dependency of the variables. 
\\ In order to build the optimality system, first of all we consider the linear state equation in its weak formulation, after a lifting procedure. Let $w= (z, q) \in \Cal Y_t$ be a test function: again we consider $w \in \Cal Y_t$ rather than $w \in \Cal Y$ in order to guarantee a future proper definition of the adjoint variable.
As we already did in \A{Section} \ref{problem}, we can specify the following forms:
\begin{align*}
 \Cal A \goesto {\Cal X}{\Cal X}{\mathbb R}, & \quad
\Cal A(x,w) = m(y,z) +\alpha n(u,v), & 
 \forall x, w \in \Cal X, \\
\Cal B: \Cal X \times \Cal Y_t \rightarrow \mathbb R, & \quad 
\Cal B (x, w; \bmu) = \intTimeSpace{\dt{y}z} + \intTime{ a((y,p), (z,q); \bmu)} - 
\intTime{c(u,z)}, &
 \forall w\in \Cal Y_t,  \\
F(\bmu) \in \Cal X\dual, & \quad
\la F(\bmu), w \ra = \intSpace{y_d(\bmu)z}, & 
 \forall w\in \Cal Y_t,\\
G \in \Cal Y_t \dual, & \quad   \la G(\pmb \mu), w \ra = 0, & \forall w \in \Cal Y_t,\\
\end{align*}
where
\begin{align*}
m \goesto {\Cal Y_t}{\Cal Y_t}{\mathbb R}, & \spazio  m(y,z) = \displaystyle  \intTimeSpace {yz}, \\
n   \goesto {\Cal U}{\Cal U}{\mathbb R}, &  \spazio n(u,v) = \displaystyle  \intTimeSpace {uv}, \\
a  \goesto {\Cal Y_t}{\Cal Y_t}{\mathbb R}, &  \spazio a((y,p), (z,q); \bmu) = \mu_1 \intSpace{\nabla{y} \cdot \nabla{z}} - \intSpace{p \dive{(z)}} -\intSpace{\dive{(y)}q},\\
c  \goesto {\Cal U}{\Cal Y}{\mathbb R}, & \spazio  c(u,q) = \displaystyle \intSpace {uq}.\\
\end{align*}
In the end, let us define the bilinear forms involved in the Stokes equations as:
\begin{align}
\label{a_Stokes}
\mathsf a  \goesto {\Cal V_t}{\Cal V_t}{\mathbb R}, &  \spazio \mathsf  a(y, z; \bmu) =
\intTime{\la \dt{y}, z \ra} + \mu_1 \intTimeSpace{\nabla{y} \cdot \nabla{z}}, \\
\label{b_continuity}
\mathsf b  \goesto {\Cal V_t}{\Cal P}{\mathbb R}, & \spazio \mathsf b(z,p) = - \intTimeSpace{p \dive{(z)}}.
\end{align}
We are now able to define the functional of the form \eqref{cal_functional} and, then, the Lagrangian functional:
\begin{equation}
\label{functional_stokes}
\Lg ((y,p),u, (\lambda, \xi); \bmu) = \Cal J((y,u); \bmu) + \Cal B(x, \sigma ; \bmu),
\end{equation}
where $\Cal J(\cdot, \cdot; \bmu)$ is defined as in \eqref{cal_functional}. After the differentiation of the functional \eqref{functional_stokes} with respect to the variables $(y, p, u, \lambda, \xi)$, we obtain the optimality system of the form 
\eqref{optimality_system}:
\begin{equation}
\label{optimality_system_Stokes}
\begin{cases}
D_y\Lg((y,p), u, (\lambda, \xi); \pmb \mu)[z] = 0 & \forall z \in \Cal V_t,\\
D_p\Lg((y,p), u, (\lambda, \xi); \pmb \mu)[q] = 0 & \forall q \in \Cal P,\\
D_u\Lg((y,p), u, (\lambda, \xi); \pmb \mu)[v] = 0 & \forall v \in \Cal U,\\
D_\lambda\Lg((y,p), u, (\lambda, \xi); \pmb \mu) [\chi]= 0 & \forall \chi \in \Cal V_t.\\
D_\xi\Lg((y,p), u, (\lambda, \xi); \pmb \mu)[\tau] = 0 & \forall \tau \in \Cal P.\\
\end{cases}
\end{equation}
Since the time dependent Stokes equations \eqref{P_Stokes} are linear, the weak formulation of the optimality system \eqref{optimality_system_Stokes} can be recast in a saddle point problem of the form \eqref{linear_optimality system}, already introduced in \A{Section} \ref{problem}. \\
Given a parameter $\bmu \in \mathscr P$, we claim that solving the optimality system \eqref{optimality_system_Stokes} is equivalent to find the solution $(x,\sigma) \in \Cal X \times \Cal Y_t$ of
\begin{equation}
\label{Stokes_optimality system}
\begin{cases}
\Cal A (x,w)  + 
	\Cal B(w, \sigma; \bmu) = 
\displaystyle \intTime{\la F(\pmb \mu), w \ra} & \forall w \in \Cal X, \\
	 \Cal B(x, \zeta; \bmu) = 
0 & \forall \zeta \in \Cal Y_t.\\
\end{cases}
\end{equation}  
We assert that the problem \eqref{Stokes_optimality system} is well-posed. In order to prove the claim, we need  the following two lemmas. 
The first one will be exploited for the \emph{inf-sup condition} of the form $\mathcal B (\cdot, \cdot; \pmb \mu)$. Indeed, the well-posedness of time dependent Stokes equations in mixed formulation implies the second hypothesis of the Brezzi's theorem for the whole state equation $\mathcal B (\cdot, \cdot; \pmb \mu)$.
\begin{lemma}
\label{lemma_infsup}
Time dependent Stokes equations with $\norm{y_0}_{L^2(\Omega(\mu_2))}=0$ in their saddle point structure verifies Brezzi's theorem.
\end{lemma}
\begin{proof}
The continuity of $\mathsf a(\cdot, \cdot; \pmb \mu) $ and of  $\mathsf b \cd $ is obvious. Moreover, the inf-sup condition for the bilinear form $\mathsf b \cd$ follows from  \cite[Theorem 4.7 and Proposition 2.2]{guberovic_schwab_stevenson_2014}. \\
It remains to prove the weakly coercivity of $\mathsf a (\cdot, \cdot; \pmb \mu)$ over the kernel of $\mathsf b \cd$. In our case, the bilinear form $\mathsf a(\cdot, \cdot; \pmb \mu)$ is actually coercive over the whole space $\Cal V_t$ since
\begin{align*}
\mathsf a(y,y; \bmu) & =  \intTime{\la \dt{y}, y \ra} + \mu_1 \intTimeSpace{\nabla{y} \cdot \nabla{y}} = \half \norm{y(T)}_{L^2(\Omega(\mu_2))}^2 + \mu_1 \norm{y}_{\Cal V}^2 \\
& \geq \mu_1 \norm{y}_{\Cal V}^2 \frac{\norm{y}_{\Cal V_t}^2}{\norm{y}_{\Cal V_t}^2}
\underbrace{\geq}_{\text{for Lemma } \ref{lemma_1_6}}\frac{\mu_1}{6}\norm{y}_{\Cal V_t}^2.
\end{align*}
\end{proof}
\cvd \\
A second lemma is needed to demonstrate the coercivity of the bilinear form $\mathcal A \cd$. Let us indicate the kernel of $\Cal B$ as $\Cal X_0$, as we did in the parabolic case: we are going to prove a norm equivalence which will be used in the proof of the well-posedness of \eqref{Stokes_optimality system}.
\begin{lemma}
\label{lemma_equiv}
On the space $\Cal X_0$, the norm $\norm{\cdot}_{\Cal X}^2$ is equivalent to $\norm{\cdot}_{\Cal V_t}^2 + \norm{\cdot}_{\Cal U}^2$. 
\end{lemma}
\begin{proof}
Let us consider $x = (y, p, u)$ in the kernel of $\Cal B$. First of all, it is obvious that 
$\norm{\cdot}_{\Cal{X}}^2 \geq \norm{\cdot}_{\Cal V_t}^2 + \norm{\cdot}_{\Cal U}^2$. Then, it remains to prove that there exists a positive constant $C_e(\bmu)$ such that 
\begin{equation}
\label{claim}
\norm{\cdot}_{\Cal{X}}^2 \leq C_e(\bmu)( \norm{\cdot}_{\Cal V_t}^2 + \norm{\cdot}_{\Cal U}^2).
\end{equation}
 If $x \in \Cal X_0$, it holds that
$$
\mathsf b(w,p) = c(u,w) - \mu_1\intTimeSpace{\nabla{y} \cdot \nabla{w}} - \intTime{\la \dt{y}, w \ra}
\spazio \forall w \in \Cal V_t.
$$
Then, we can derive the following inequalities for all $w \in \Cal V_t$:
\begin{align*}
\mathsf b(w,p) & \leq \norm{u}_{\Cal U} \norm{w}_{\Cal V} 
+ \mu_1\norm{y}_{\Cal V}\norm{w}_{\Cal V} + \parnorm{\dt{y}}_{\Cal V \dual} \norm{w}_{\Cal V} \\
& \leq \min \{ 1, \mu_1 \} (\norm{y}_{\Cal V_t} + \norm{u}_{\Cal U})\norm{w}_{\Cal V_t}.
\end{align*}
Now,  since the above inequality does not depend from $w$ and $p$, from lemma \ref{lemma_infsup} we know that
$$
\bar{\beta} \norm{p}_{\Cal P}\norm{w}_{\Cal V_t}\leq \inf_{0\neq p \in \Cal P}
\sup_{0 \neq w \in \Cal V_t} \mathsf b(w,p) \leq \min \{ 1, \mu_1 \} (\norm{y}_{\Cal V_t} + \norm{u}_{\Cal U})\norm{w}_{\Cal V_t}.
$$
Moreover, applying Young's inequality, we have that
\begin{align*}
\norm{p}_{\Cal P}^2 \leq \frac{\min \{1, \mu_1\}^2}{\bar \beta^2}(\norm{y}_{\Cal V_t} + \norm{u}_{\Cal U})^2
\leq 2\frac{\min \{1, \mu_1\}^2}{\bar \beta^2}(\norm{y}_{\Cal V_t}^2 + \norm{u}_{\Cal U}^2).\\
\end{align*}
Setting 
$\displaystyle {\bar{C}}_e(\bmu) = \Big ( 2\frac{\min \{1, \mu_1\}^2}{\bar \beta^2}\Big )$ the inequality \eqref{claim} is verified with $C_e(\bmu) = \min\{1, {\bar{C}}_e(\bmu) \}$.
\end{proof}
\cvd \\
\no Thanks to these two lemmas, we are now able to demonstrate the following theorem, which guarantees existence and uniqueness of the optimal solution for \ocp s governed by time dependent Stokes equations.
\begin{theorem}
The problem \eqref{Stokes_optimality system} is equivalent to the minimization \A{problem} \eqref{func_Stokes} under the \fb{constraint} \eqref{P_Stokes}. Moreover, the saddle point problem \eqref{Stokes_optimality system} admits a unique solution.
\end{theorem}
\begin{proof}
In order to prove existence and uniqueness of the solution of problem \eqref{Stokes_optimality system} and its equivalence with the minimization problem \eqref{func_Stokes}-\eqref{P_Stokes}, we exploit \A{Proposition} \ref{prop}, i.e. we show that:
\begin{itemize}
\item $\Cal A \cd$ is a symmetric positive definite continuous bilinear form
 coercive on $\Cal X_0$;
\item $\Cal B (\cdot, \cdot; \pmb \mu)$ is continuous and satisfies the \emph{inf-sup condition}.
\end{itemize}
The form $\Cal A \cd$ is trivially symmetric and positive definite thanks \fb{to} its definition. The continuity of the bilinear forms $\Cal A \cd$ and $\Cal B (\cdot, \cdot; \pmb \mu)$ can be derived using the same techniques already presented in theorem \ref{Brezzi_continuous}. Let us now focus on the coercivity of $\Cal A\cd$. For every $x \in \Cal X_0$ we want to prove that there exists a positive constant C such that
$$
\Cal A (x,x) \geq C \norm{x}_{\Cal X}^2.
$$
Since norm of $\norm{\cdot}_{\Cal X}^2$ is equivalent to $\norm{\cdot}_{\Cal V_t}^2 + \norm{\cdot}_{\Cal U}^2$ as we proved in Lemma \ref{lemma_equiv}, it is sufficient to show that
$$
\Cal A (x,x) \geq C (\norm{\cdot}_{\Cal V_t}^2 + \norm{\cdot}_{\Cal U}^2).
$$
When $x \in \Cal X_0$ it means that the \fb{pair} $(y,p)$ verifies the Stokes equations with the control as forcing term. It is well known that the velocity solution of Stokes equations with $\norm{y_0}_{L^2(\Omega(\mu_2))}=0$  and control as forcing term admits a $k(\bmu)>0$ such that 
$\norm{y}_{\Cal V_t} \leq k(\bmu)\norm{u}_{\Cal U}$, see \cite[Chapter 13]{quarteroni2008numerical}. Then for $x \in \Cal X_0$
\begin{align*}
\Cal A (x,x) & \geq \norm{y}_{L^2(0,T; H)}^2 + \alpha \norm{u}_{\Cal U}^2
\geq \alf \norm{u}_{\Cal U}^2 + \alf \norm{u}_{\Cal U}^2
\geq \frac{\alpha}{2k(\bmu)^2}\norm{y}_{\Cal V_t} ^2 + \alf \norm{u}_{\Cal U}^2 \\
& \geq \min \Big \{   \frac{\alpha}{2k(\bmu)^2}, \alf \Big \} (\norm{y}_{\Cal V_t}^2 + \norm{u}_{\Cal U}^2).
\end{align*}
It remains to show that
\begin{equation}
\label{claim_infsup}
\inf_{0 \neq q \in \Cal Y_t}
\sup_{0 \neq x \in \Cal X } \frac{\Cal B(x, q; \bmu)}{\norm{x}_{\Cal X}\norm{q}_{\Cal Y_t}} > \beta(\bmu) > 0.
\end{equation}
First of all, let us consider the bilinear form 
$\mathsf A \goesto {\Cal Y_t}{\Cal Y_t}{\mathbb R}$ defined as 
$\mathsf A((y,p), (w,q); \bmu) = \mathsf a(y,w) + \mathsf b(w,p) + \mathsf b(v,q)$. We can now proceed as presented in \cite[Appendix A.1]{negri2015reduced}.
Thanks to lemma \ref{lemma_infsup}, we infer that the operator $\mathsf A (\cdot, \cdot; \bmu)$ is a well-posed mixed problem, then the Babu$\check{s}$ka inf-sup constant $\beta_B(\bmu)$ is well defined as

\begin{equation}
\beta_B(\bmu) = \inf_{0 \neq (y,p) \in \Cal Y_t}
\sup_{0 \neq (w,q)\in \Cal Y_t} \frac{\mathsf A((y,p), (w,q); \bmu)}{\norm{(y,p)}_{\Cal Y_t}\norm{(w,q)}_{\Cal Y_t}} =
\inf_{0 \neq (w,q) \in \Cal Y_t}
\sup_{0 \neq (y,p)\in \Cal Y_t} \frac{\mathsf A((y,p), (w,q); \bmu)}{\norm{(y,p)}_{\Cal Y_t}\norm{(w,q)}_{\Cal Y_t}}, \\
\end{equation}
see the classical reference \cite{Babuska1971}. Then we have the following inequalities:
\begin{align*}
\sup_{0 \neq x \in \Cal X } \frac{\Cal B(x, \sigma; \bmu)}{\norm{x}_{\Cal X}\norm{\sigma}_{\Cal Y_t}}
& \underbrace{\geq}_{x = (y,p,0)}  \sup_{0 \neq (y,p)\in \Cal Y_t} \frac{\mathsf A((y,p), (w,q); \bmu)}{\norm{(y,p)}_{\Cal Y_t}\norm{(w,q)}_{\Cal Y_t}} \\
& \geq
\inf_{0 \neq (w,q) \in \Cal Y_t}
\sup_{0 \neq (y,p)\in \Cal Y_t} \frac{\mathsf A((y,p),(w,q); \bmu)}{\norm{(y,p)}_{\Cal Y_t}\norm{(w,q)}_{\Cal Y_t}} \geq \beta_B(\bmu) \\
\end{align*}
Since the above inequality does not depend on the choice of $\sigma \in \Cal Y_t$, then the inequality \eqref{claim_infsup} is verified and the \ocp $\;$ is well-posed. \\
\cvd 
\end{proof}
\no In this section we generalized the saddle point theory of \ocp s governed by parabolic state equations to \ocp s under Stokes equations constraint. After the introduction of the minimization problem, we provided a proof for the well-posedness of the \ocp $\;$ in the saddle point framework. The next \A{Section} will focus on the full order discretization of the \ocp $\;$ proposed. 
\subsection{Finite Element Discretization and Time Discretization for \ocp s governed by Stokes Equations: all-at-once approach}
\label{FE_problem_Stokes}
In this section we introduce the discretized version of time dependent \ocp s governed by Stokes equations following the properties of the schemes already presented in \cite{HinzeStokes,Stoll}.\\ 
\no In order to expose the discrete formulation of the optimality system \eqref{optimality_system_Stokes}, we need its strong formulation. The minimization of the functional \eqref{func_Stokes} constrained to time dependent Stokes equations \eqref{P_Stokes} is equivalent to find the \fb{pair} 
$(x, \sigma) \in \Cal X \times \Cal Y_t$ such that the state, the adjoint and the optimality equations read as follows: 
\begin{equation}
\label{Stokes_Strong}
\begin{cases}
\displaystyle \dt{y} - \mu_1 \Delta y + \nabla p = u & \text{\spazio in } \Omega(\mu_2) \times \B{(0,T)} , \\
\dive(y) = 0 & \text{\spazio in } \Omega(\mu_2) \times \B{(0,T)} , \\
y(t) = g(t) & \text{\spazio on  } \Gamma_D(\mu_2) \times  \B{(0,T)}, \\
\displaystyle \dn{y} = 0& \text{\spazio on  } \Gamma_N (\mu_2) \times  \B{(0,T)}, \\
y(0) = y_0 & \text{\spazio in } \Omega(\mu_2) \times  \{0\} , \\
\displaystyle y - \dt{\lambda} - \mu_1 \Delta \lambda + \nabla \xi = y_d & \text{\spazio in } \Omega(\mu_2) \times \B{(0,T)} , \\
\dive(\lambda) = 0 & \text{\spazio in } \Omega(\mu_2) \times \B{(0,T)},  \\
\lambda(t) = 0 & \text{\spazio on  } \partial \Omega(\mu_2) \times \B{(0,T)}  , \\
\lambda(T) = 0 & \text{\spazio in }\Omega(\mu_2) \times  \{T\} , \\
\alpha u = \lambda & \text{\spazio in } \Omega(\mu_2) \times \B{(0,T)}, \\
\text{boundary conditions} & \text{\spazio on $\partial{\Omega} \times \B{(0,T)}.$}
\end{cases}
\end{equation}
In order to discretize in time and in space the system \eqref{Stokes_Strong}, we will exploit the technique already presented in \A{Section} \ref{Algebraic_parabolic} for parabolic time dependent \ocp s. 
 \\ Let us define $\pmb y = [y_1, \dots, y_{N_t}]^T, \pmb p = [p_1, \dots, p_{N_t}]^T$, $\pmb u = [u_1, \dots, u_{N_t}]^T$ and $\pmb \lambda= [\lambda_1, \dots, \lambda_{N_t}]^T$, $\pmb \xi = [\xi_1, \dots, \xi_{N_t}]^T$, i.e. the vectors where $y_i, u_i$, $p_i$, $\lambda_i$ and $\xi_i$ for $1 \fb{\leq i} \leq N_t$ are \A{row vectors} representing the coefficients of the FE discretization of each time instance. Also in this case we divided the time interval in $N_t$ sub-intervals of length $\Delta t$.  The vector representing the initial condition for the state variable is 
$\pmb y_0 = [y_0, 0, \dots, 0]^T$. The discretized desired state and the forcing term are given by $\pmb y_d = [y_{d_{1}}, \dots, y_{d_{N_t}}]^T$ and $\pmb g = [g_1, \dots, g_{N_t}]^T$, respectively.
\\  In order to create the time steps, we used the rectangle composite quadrature formula following the same arguments as already presented in \A{Remark} \ref{time_int}.
The discretized state equation reads as follows:
\begin{equation}
\begin{cases}
My_k + \mu_1 \Delta t Ky_k + \Delta t  D^T p_k = \Delta t M u_k  +My_{k-1}  + \Delta t g_k & \B{\text{for }k \in 
\{1, 2, \dots, N_t\}}, \\
Dy_k = 0 &  \B{\text{for }k \in 
\{1, 2, \dots, N_t\}}, \\
\end{cases}
\end{equation} 
where $M$ and $K$ are the mass and the stiffness matrices relative to the FE discretization, respectively. Moreover, $D$ is the differential operator representing to the continuity equation. Then, in order to solve the system in an all-at-once approach we have to solve the following system:
\[
\Cal K [\pmb y, \pmb p] - \Delta t \Cal C \pmb u = \Cal M \pmb y_0 + \Delta t \pmb g,
\]
where:
\[
\Cal K =
\begin{bmatrix}
  M - \mu \Delta t K & 0  & \cdots & 0 & \Delta t D^T & 0 & \cdots &  &  & 0 \\ 
 D & 0  &  &  &  &  &  &  &  & \\ 
 -M &  M - \mu \Delta t K& 0  & \cdots  & 0 & \Delta t D^T & 0 & \cdots & & 0 \\ 
0 & D &  &  &  &  &  &  &  &  \\ 
&   \ddots& \ddots & & &  &  &  &  &   \\ 
&  &  \ddots& \ddots & &  &  &  &  &   \\ 
 &  &  &\ddots& \ddots &  &  &  &  &   \\ 
 &  & & &    \ddots & \ddots &  & & &   \\ 
 0 & \cdots &  &  0 & -M &  M - \mu \Delta t K& 0 & \cdots & 0 & \Delta t D^T \\ 
 0 &   \cdots &  &  &  &  D &  & 0& \cdots& 0
\end{bmatrix}
\]
and $\Cal C = \Cal M \in \mathbb R^{\Cal N \cdot N_t} \times \mathbb R^{\Cal N \cdot N_t}$ is the \A{block} diagonal matrix which \A{diagonal} entries are $[M, \cdots, M]$. \\
Then, we discretize the optimality equation, which has the following form:
\begin{align*}
\alpha \Delta t M u_k - \Delta t M \lambda_k = 0 & \spazio \B{\text{for } k \in \{1,2, \dots, N_t \}}.
\end{align*} This means that the system considered is of the form:
$$
\alpha \Delta t \Cal M \pmb u - \Delta t \Cal C^T  \pmb \lambda = 0.
$$
Finally, we are able to treat the adjoint equation. We use a forward Euler as we did for the parabolic case in \A{Section} \ref{Algebraic_parabolic}. Thanks to this approach, it holds:
\begin{equation}
\begin{cases}
M \lambda_{k}  = M \lambda_{k+1}  + \Delta t ( - M y_{k} + \mu_1 \Delta t K\lambda_{k}  - \Delta t  D^T \xi_{k} + My_{d_{k}} )
& \B{\text{for } k \in \{ N_t - 1, N_t -2, \dots, 1\}}, \\
D\lambda_{k} = 0  & \B{\text{for } k \in \{ N_t - 1, N_t -2, \dots, 1\}}.
\end{cases}
\end{equation}
In the end, the adjoint equation has the following form:
$$
\Delta t {\Cal M} \pmb y + \Cal {K}^T [\pmb \lambda, \pmb \xi] = \Delta t \Cal M \pmb y_d.
$$
Then, the final parametrized system to be solve trough a one shot approach is the of the form presented in \A{\eqref{one_shot_system}}, i.e.:
\begin{equation}
\label{Stokes_one_shot_system}
\begin{bmatrix}
\Delta t {\Cal M} & 0 & \Cal {K}^T  \\
0 & \alpha \Delta t \Cal M &  -\Delta t \Cal C^T \\
\Cal K & -\Delta t \Cal C & 0 \\
\end{bmatrix}
\begin{bmatrix}
 (\pmb y,\pmb p) \\ 
\pmb u \\
(\pmb \lambda,\pmb \xi) \\
\end{bmatrix} 
= 
\begin{bmatrix}
\Delta t {\Cal {M}} \pmb y_d  \\
0 \\
 {\Cal {M}} \pmb y_0 +\Delta t \pmb g \\
\end{bmatrix}.
\end{equation}

\no Now, calling 
\begin{align*}
A = 
\begin{bmatrix}
\Delta t {\Cal M} & 0 \\
0 & \alpha \Delta t \Cal M \\
\end{bmatrix},
\quad B = 
\begin{bmatrix}
\Cal K & -\Delta t \Cal C \\
\end{bmatrix},
\quad F =
\begin{bmatrix}
\Delta t {\Cal {M}} \pmb y_d  \\
0 \\
\end{bmatrix}
\quad \text{and } \quad
G = {\Cal {M}} \pmb y_0 +\Delta t \pmb g
\end{align*}
the system \eqref{one_shot_system} can be written as follows:
\begin{equation}
\label{Stokes_one_shot_system_saddle_point}
\begin{bmatrix}
A & B^T \\
B & 0 \\
\end{bmatrix}
\begin{bmatrix}
\pmb x \\
\pmb \chi
\end{bmatrix} 
=
\begin{bmatrix}
F \\
G
\end{bmatrix},
\end{equation}
where $\pmb x = ((\pmb y, \pmb p), \pmb u) $ and $\pmb \chi= (\pmb \lambda, \pmb \xi)$. \\
As in the parabolic case, we used a direct approach in order to find the solution to the saddle point system
\eqref{Stokes_one_shot_system_saddle_point}, even if iterative solvers are widely exploited for this kind problem: the interested reader may refer to  \cite{HinzeStokes,schoberl2007symmetric,Stoll}. \\
We now move towards ROM version of the system, focusing on supremizers aggregated space technique \cite{negri2015reduced}  which guarantees the well-posedness of the reduced optimality system.
\subsection{Reduced basis method: Supremizer Stabilization and Aggregated Spaces}
\label{aggregated_space_Stokes}
In this section, we exploit a POD-Galerkin strategy in order to recast the full order optimality system in a reduced framework. The technique used is the same already presented in \A{Section} \ref{PODsec}: we remark that in this case we apply the POD algorithm for each space-time FE variables
$v\disc, p\disc, u\disc, \lambda \disc$ and $\xi \disc$, which contain all the temporal instances. \\ Let us suppose to have built the reduced spaces with the POD approach, then the reduced optimality system reads: given $\pmb \mu \in \mathscr P$, find $(x_N(\pmb \mu), \sigma _N (\pmb \mu)) \in \Cal X_N \times \Cal Y_N$ such that
\begin{equation}
\label{CarolineStoRB}
\begin{cases}
\Cal A(x_N (\pmb \mu), w_N) + \Cal B(w_N, \sigma_N(\bmu); \pmb \mu) = \la F(\pmb \mu), w_N \ra & \forall w_N\in \Cal X_N, \\
\Cal B(x_N(\pmb \mu), \zeta_N; \pmb \mu) = \la G(\pmb \mu), \zeta_N\ra & 
\forall \zeta_N \in \Cal Y_N.
\end{cases}
\end{equation}
Anyway, in this case we are managing  a nested saddle point structure, since we are dealing with Stokes equations as constraints. Indeed, the state equation problem is formulated as: 
given $\pmb \mu \in \mathscr P$, find $(y_N(\pmb \mu), p_N (\pmb \mu)) \in \Cal V_{t_N} \times \Cal P_N$ such that
\begin{equation}
\label{StokesROM}
\begin{cases}
\mathsf a(y_N (\pmb \mu), z_N; \pmb \mu) + \mathsf b(z_N, p_N(\bmu)) = 0 & \forall z_N\in \Cal V_{t_N}, \\
\mathsf b(y_N(\pmb \mu), \zeta_N) = 0 & 
\forall \zeta_N \in \Cal P_N,
\end{cases}
\end{equation}
where the bilinear forms $\mathsf a(\cdot, \cdot; \pmb \mu)$ and $\mathsf b \cd$ are defined as \eqref{a_Stokes} and \eqref{b_continuity}, respectively, and $\Cal V_{t_N}$ and $\Cal P_N$ are the reduced spaces obtained trough POD algorithm over velocity and pressure variables. It is well known that, in order to guarantee the well-posedness of the reduced Stokes state equations \eqref{StokesROM}, the \emph{reduced inf-sup stability \A{condition}} is required for the \A{bilinear form} $\mathsf b(\cdot, \cdot)$ and it is not directly inherited from the FE approximation. In other words, we have to build our reduced spaces for velocity and pressure variables such that there exists a $\bar{\beta}_N > 0$ which verifies
\begin{equation}
\label{infsup_stokes}
\inf_{p_N \in \Cal P_N} \sup_{y_N \in {{\Cal V}_t}_N}
\frac{\mathsf  b(y_N, p_N; \pmb \mu)}{ \norm {y_N}_{{\Cal V}_t}\norm {p_N}_{\Cal P}} \geq \bar{\beta}_N (\pmb \mu)  > 0 \hspace{1cm} \forall \pmb \mu \in \mathscr P.
\end{equation}
To ensure the inequality \eqref{infsup_stokes} at the reduced level, we follow the strategy of pressure supremizers \cite{rozza2007stability}: let us consider the supremizer operator
$T^{\pmb \mu}_p: \Cal P \disc \rightarrow{{\Cal V}_t} \disc$ defined as follows:
\begin{equation}
(T^{\pmb \mu}_p s, \phi)_{\Cal{V}_t} = \mathsf  b(\phi, s; \pmb \mu), \hspace{1cm} \forall \phi \in {\Cal V}_t \disc.
\end{equation}
Then, we enrich the reduced velocity space with supremizers and we build a new space as follows:
$$
{{\Cal V}_t}_N ^{\pmb \mu} = \text{span}\{ y \disc (\pmb \mu^n), \; T^{\pmb \mu^n}_p p\disc(\pmb \mu^n), \; n = 1,\dots,N\}.
$$
If we exploit this new reduced space ${{\Cal V}_t}_N ^{\pmb \mu}$ for velocity, then \eqref{infsup_stokes} is verified.
\\Once proved the stability of the Stokes equations, we can focus on the whole OCP($\pmb \mu$) governed by this particular state equation. In order verify the \emph{reduced inf-sup condition} for the bilinear form 
$\Cal B(\cdot, \cdot; \pmb \mu)$, we used the aggregated spaces strategy already presented in \A{Section} \ref{agg}. As shown in \cite{negri2015reduced}, we define the aggregated spaces for both the state and adjoint pressure variables as
\begin{equation}
\Cal P_N = \text{span}\{p \disc (\pmb \mu^n), \xi \disc (\pmb \mu^n), \; n = 1, \dots, N\},
\end{equation}
\no while, for the state and the adjoint velocity variables, we consider
\begin{equation}
{{\Cal V}_t}_N^{\pmb \mu^n} = \text{span } \{y \disc (\pmb \mu^n), T^{\pmb \mu^n}_p p\disc(\pmb \mu^n), 
\lambda \disc (\pmb \mu^n), T^{\pmb \mu^n}_p \xi \disc(\pmb \mu^n), \; n = 1, \dots, N\}.
\end{equation}
Let us now define the following \emph{aggregated} space for state and adjoint velocity-pressure variables
$$
\Cal Y_N = {{\Cal V}_t}_N^{\pmb \mu^n} \times \Cal P_N.
$$
\no Finally, the control space is:
\begin{equation}
\Cal U_N = \text{span }\{u \disc(\pmb \mu^n), \; n = 1, \dots, N\}.
\end{equation}
Considering the product space $\Cal X_N = \Cal Y_N \times \Cal U_N$, the well-posedness of the reduced optimality system \eqref{CarolineStoRB} in its nested saddle point structure is verified.
\no The stability of these techniques is numerically demonstrated in \cite{ballarin2015supremizer,gerner2012certified,negri2015reduced,RozzaHuynhManzoni2013,rozza2012reduction}. Thanks to this new reduced spaces formulation, the dimension of the state and the adjoint spaces is $6N$, whereas the control space has dimension $N$, for a total reduced dimension $N_{\text{tot}}=13N$.
In the next section, we are going to test this methodology in a numerical \A{simulation}, showing that ROM online phase could be of great advantage, despite the increasing value of $N_{\text{tot}}$.

\section{Numerical Results: Time Dependent \ocp $\;$ for a Cavity Viscous Flow}
In this section we propose a geometrical and physical parametrized version of the \A{test case} already presented in \cite{HinzeStokes,Stoll}: it is a time dependent \ocp $\,$ for cavity viscous flow problem. The parametrization presented has been studied also in \cite{ali2018stabilized}, where several stabilization techniques are proposed. In this case, we deal with the consistent FE pair $\mathbb P^2-\mathbb P^1$  for velocity and pressure variables, respectively: then, no stabilization is needed except for the supremizer techniques described in \A{Section} \ref{aggregated_space_Stokes}.  For the discretization of the control variable, we exploited $\mathbb P^2$ polynomials for its FE approximation. \B{We call $\Cal N_y$, $\Cal N_p$ and $\Cal N_u$ the FE dimensions for velocity, pressure and control space, respectively. Our space discretization leads to $\Cal N_y = \Cal N_u = 4554$ and $\Cal N_p = 591$. Then we define the global FE dimension as $\Cal N = 2\Cal N_y + 2\Cal N_p + \Cal N_u$.}
The \ocp $\;$ is based on the minimization of the functional \eqref{func_Stokes} with \eqref{P_Stokes} as state equation. We now provide the information about the \A{numerical test}, which are resumed in \A{Table} \ref{table_data_Stokes}. First of all, let us specify the role of the parameters: $\mu_1$ is physical parameter describing the diffusivity action of the system, while $\mu_2$ changes the geometry of the problem, stretching the length of the domain. In order to deal with the geometrical parametrization, we traced back $\Omega(\mu_2) \subset \mathbb R^2$ into a two dimensional square \B{$ (0,1) \times (0,1)$} as reference domain,  where $\Gamma_{IN} = \B{(0,1)}\times \{1 \}$ is the inlet boundary and $\Gamma_D = \partial \Omega \setminus \Gamma_{IN}$ is characterized by homogeneous Dirichlet boundary conditions. The reference domain structure is presented in \A{Figure} \ref{domain_Stokes}. 
 Since we considered Dirichlet boundary conditions all over the domain, we assumed $p \in \Cal P = L^2(0,T; L^2_0(\Omega(\mu_2)))$\footnote{
The space $L^2_0(\Omega(\mu_2))$ is made by functions  $p \in L^2(\Omega(\mu_2))$ which satisfy 
$\displaystyle \intSpace p = 0$. In the reduced model the aggregated basis associated to the state and adjoint pressure was built in order to satisfy this constraint, i.e. the reduced adjoint variable $\xi_{N}$ verifies 
$\displaystyle \intSpace {\xi_{N}}= 0$.  
}. The aim of the \ocp $\,$ proposed is to make the state velocity $y$ the most similar to a target velocity $y_d$, given $ \pmb \mu \in \mathscr P = [10^{-3}, 10^{-1}] \times [0.5, 2.5] $. The target velocity profile $y_d$ is defined as the FE solution of the uncontrolled time dependent Stokes equations \A{with Dirichlet boundary conditions given by the constant velocity components $(1,0)$ in $\Gamma_{IN}$ and homogeneous Dirichlet boundary conditions in $\Gamma_D$, for
$\mu_1 = 1$ fixed}. The target velocity has been simulated in $\B{(0,T)} = \B{(0,1)}$. 
\no For the \ocp $\,$ test case, we consider a different time dependent inlet  boundary condition: the state velocity profile is 
$\displaystyle y = \Big (1 + \half \cos (4\pi t - \pi), 0 \Big )$ over $\Gamma_{IN}$ for $t \in \B{(0,T)} = \B{(0,1)}$. As already specified, homogeneous Dirichlet boundary conditions are considered elsewhere. The distributed control $u$ has the role to reduce the impact of the periodic inlet over the system.
\no The FE discretization is performed as presented in \A{Section} \ref{FE_problem_Stokes}. For the time discretization we used $\Delta t = 0.05$, resulting in $N_t = 20$ time instances. The total number of degree of freedom is \B{$\Cal N_{\text{tot}} = N_t \times \Cal N = 296'880$}. In order to reduce the dimension of the FE system, we applied the POD-Galerkin algorithm described in \A{Section} \ref{PODsec}. For all the variables involved, we choose $N_{\text{max}} = 70$ snapshots for the correlation matrix and we only retain the first $N = 25$ eigenvectors as basis functions of our low-dimensional spaces. Once again, we will call $N$ the \emph{basis number}. As already specified in \A{Section} \ref{ADR_OCP}, the choice of the number of snapshots $N_{\text{max}}$ is strongly related to the complexity of the offline phase of the time dependent optimality system, which turned out to be very expensive to be solved. The total dimension of the reduced system is $N_{\text{tot}}= 13N = 325$, tanking into account supremizer aggregated space technique for the nested saddle point problem given by the Stokes equations and the optimal control structure, as already specified in \A{Section} \ref{aggregated_space_Stokes}. A comparison between FE and ROM state velocity and state pressure profiles is shown in \A{Figure} \ref{simulation_Stokes}: the ROM simulations recover the behaviour of the FE solutions for different time instances. The accuracy of the method is also underlined by the plots in \A{Figure} \ref{error_Stokes}, where the average relative error\footnote{The error for state velocity and pressure, control and adjoint velocity and pressure variables are: $\norm{y\disc - y_N}^2_{H^1}$, $\norm{p\disc - p_N}^2_{L^2}$, $\norm{u\disc - u_N}^2_{L^2}$, $\norm{\lambda \disc - \lambda_N}^2_{H^1}$ and $\norm{\xi\disc - \xi_N}^2_{L^2}$, respectively. We underline that in order to make the FE and ROM adjoint pressures comparable we define 
$\xi \disc$ := $\bar \xi \disc - \displaystyle \int_{\Omega} \bar \xi \disc dx$, where $\bar \xi \disc$ is the actual truth solution.} plots
over a testing set of 35 \B{uniformly distributed} parameters are presented: the greater is the value of $N$ the better are the results, as expected. The relative error is about $10^{-3}$ for all the involved variables. In \A{Figure} \ref{error_Stokes}, we also show the relative error between the value of the functional \emph{output} \eqref{func_Stokes} evaluated for the FE solution and for the ROM solution: it reaches the very low value of $10^{-8}$. 
Also in this test case, ROM is a strategy which reaches accurate results in a small computational time. In \A{Figure} \ref{speedup_Stokes} we show the \emph{speedup} index for this test case: the POD-Galerkin approach could be very useful for time dependent \ocp s governed by Stokes equations, since it guarantees a good approximation of a quite complex system in low-dimensional framework. The number of reduced problem one can perform in a single FE simulation decreases with respect to the value of $N$, but still it is never below the value of $6 \cdot 10^{4}$. Also in this case, from the high speedup values, we can conclude that ROM is a good approach to manage several repeated simulations for different values of the parameter considered.
\begin{figure}[H]
\hspace{-2cm}\includegraphics[scale = .3]{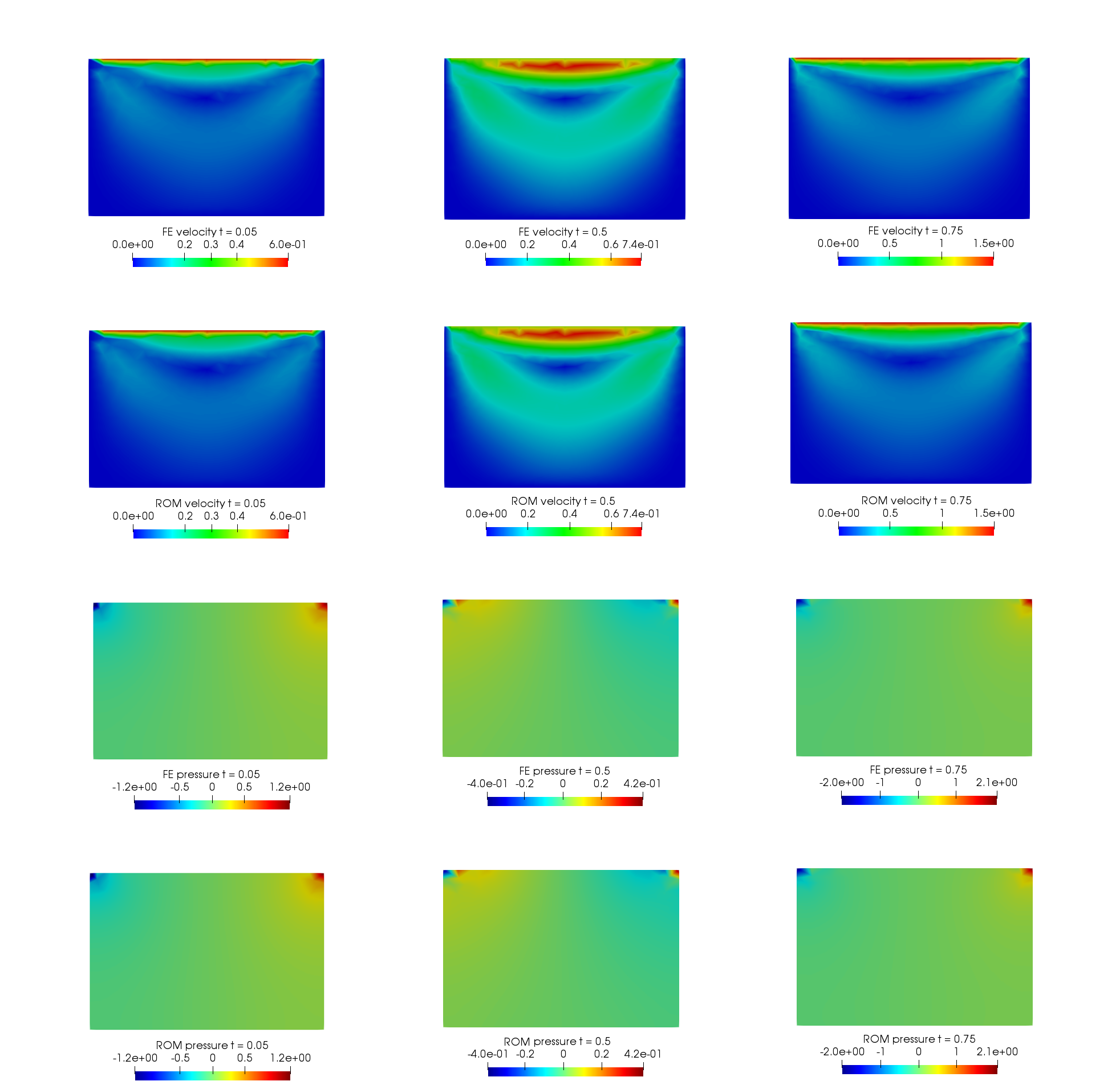}
\caption{FE vs ROM simulations of state velocity and pressure  variables for $\pmb \mu = [1/2, 10^{-2}]$ and $\alpha = 10^{-2}$. First row: FE state velocity for $t=0.05,0.5,0.75$; second row: ROM state velocity for $t=0.05,0.5,0.75$; third row: FE state pressure for $t=0.05,0.5,0.75$; fourth row: ROM state pressure for $t=0.05,0.5,0.75$.}
\label{simulation_Stokes}
\end{figure}
\begin{figure}[H]
\begin{center}
\includegraphics[scale = .25]{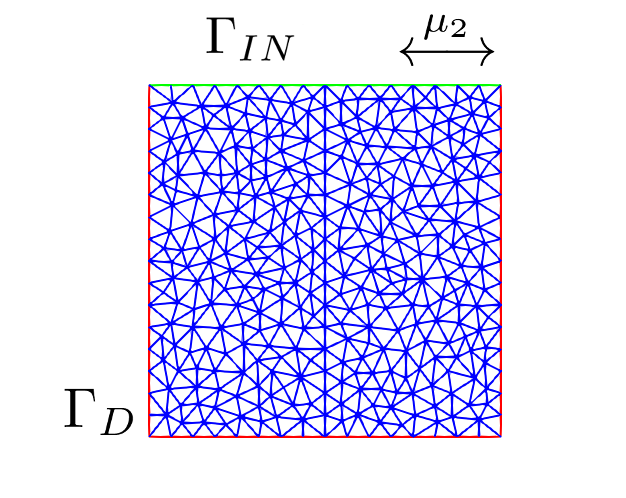}
\caption{Domain and boundary considered.}
\label{domain_Stokes}
\end{center}
\end{figure}
\begin{center}
\begin{table}[H]
\caption{Data for the \ocp  $\:$ governed by a Stokes equation.}
\label{table_data_Stokes}
\centering
\begin{tabular}{ c | c }
\toprule
\textbf{Data} & \textbf{Values} \\
\midrule
$\mathscr P$ & $[0.5, 2.5] \times [10^{-3}, 10^{-1}]$\\
\midrule
 $(\mu_1, \mu_2, \alpha)$ &  $(10^{-2}, 3/2, 10^{-2})$ \\
\midrule
\A{$N_{\text{max}}$} & 70 \\ 
\midrule
Basis Number $N$ & 25 \\
\midrule
Sampling Distribution & Uniform \\
\midrule
{$\Cal N_{\text{tot}}$} & $\B{296'880}$ \\
\midrule
\A{$ N_{t}$} & \A{$20$} \\
\midrule
ROM System Dimension & $325$ \\
\bottomrule
\end{tabular}
\end{table}
\end{center}

\begin{figure}[H]
\begin{center}
\includegraphics[scale = .33]{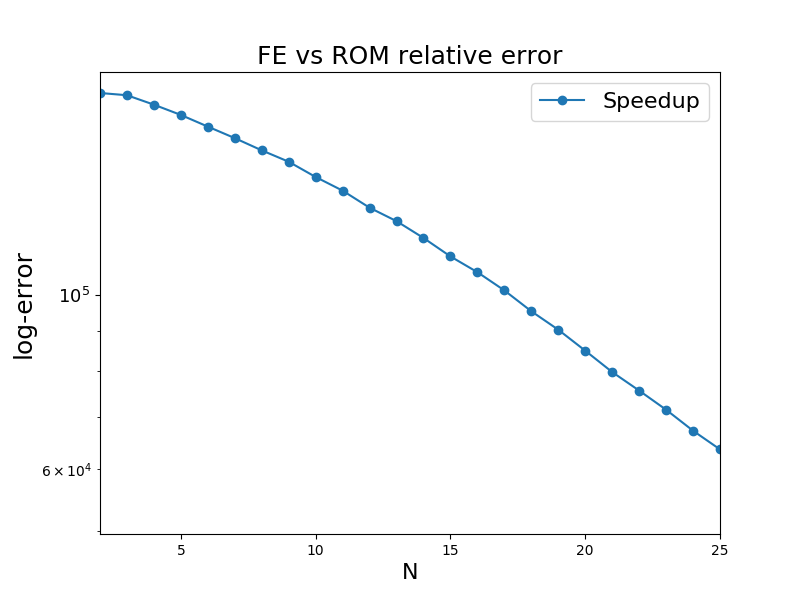}
\caption{Speedup analysis in logarithmic scale shown with respect to the reduced basis dimension $N$.}
\label{speedup_Stokes}
\end{center}
\end{figure}
\begin{figure}[H]
\begin{center}
\includegraphics[scale = .33]{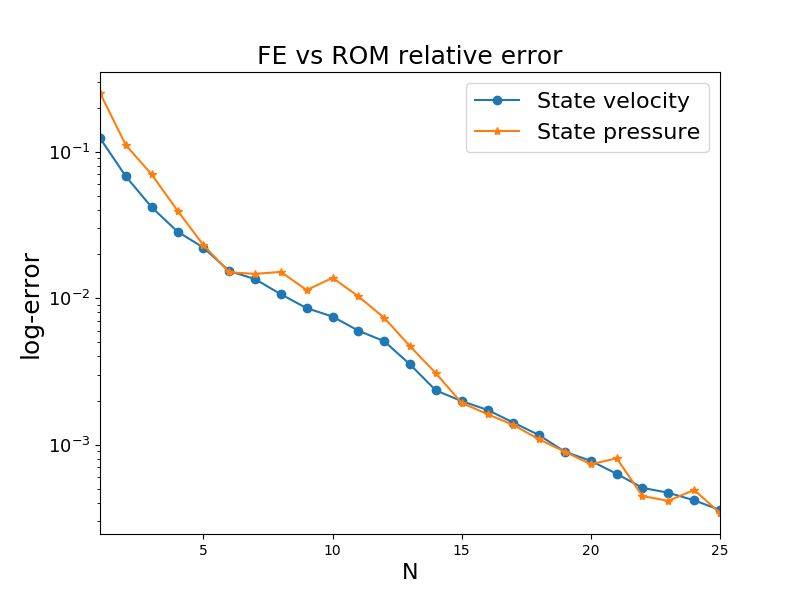}\includegraphics[scale = .33]{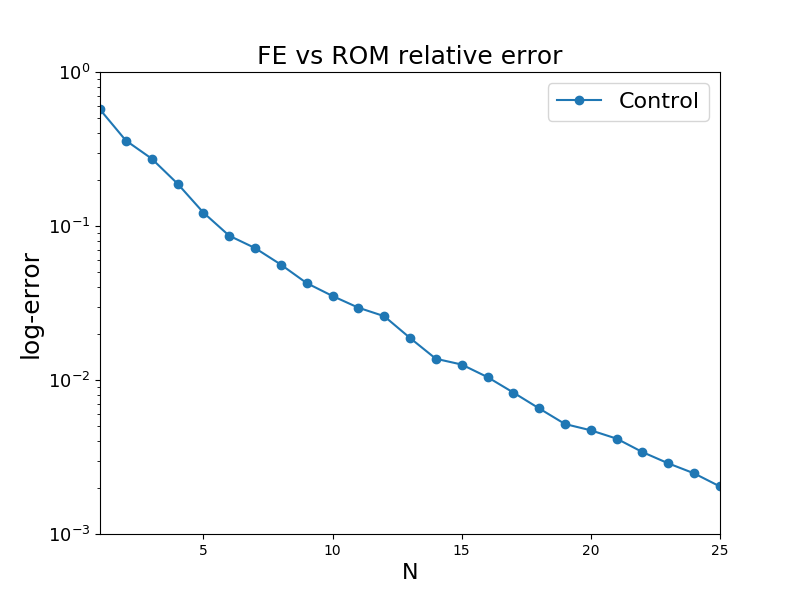}\\
\includegraphics[scale = .33]{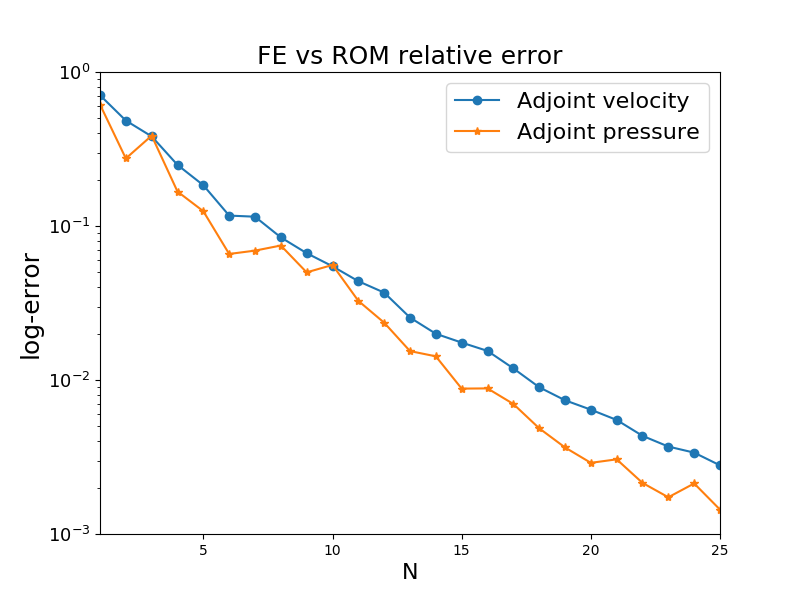}\includegraphics[scale = .33]{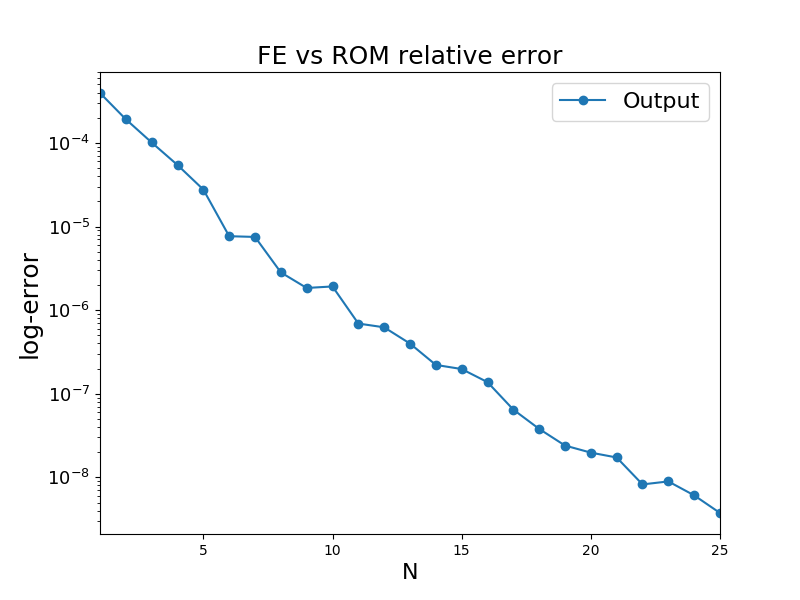}
\vspace{-5mm}
\caption{FE vs ROM errors in logarithmic scale. Top left: state error (velocity and pressure), top right: control error, bottom left: adjoint error (velocity and pressure), bottom right: output error.}
\label{error_Stokes}
\end{center}
\end{figure}




\section{Conclusions and Perspectives}
\label{conc}
In this work we have exploited ROM in \ocp s dealing with time dependent linear state equations. First of all we showed how the saddle point structure for linear quadratic \ocp $\:$ is maintained also for the time dependent case. Then, we have exploited a POD-Galerkin method as sampling strategy  for the projection of the systems in a low dimensional framework in order to solve them in an accurate and fast way. The results have been tested, on one side, for a boundary \ocp $\,$ governed by a Graetz flow, on the other side, the test case has been performed over an \ocp $\,$ constrained to time dependent Stokes equations. To the best of our knowledge, the main novelty of this work is in the POD-Galerkin reduction used for this all-at-once time dependent optimal control problems recast in saddle point formulation: it is a very versatile approach due to the great speedup index, for both the state equations analyzed.\\
Some improvements of this work \A{will} follow. \both{First of all, we would like to stress that the used all-at-once approach is not the most appropriate. Indeed, it leads to a very costly formulation and complex linear systems to deal with. This issue can be overcome thanks to more computing power and/or preconditioned iterative strategies as proposed in \cite{benzi_golub_liesen_2005,schoberl2007symmetric,Stoll1,Stoll}. Some improvements are needed in this direction, since the direct solution of such a system affects not only the resolutions of full order solutions and, consequently, of the reduced ones, but also the time needed for the offline phase. } \BB{ Indeed, the proposed all-at-once formulation is the major reason for an expensive basis construction. Even if we accomplished our goal of showing how ROMs can be very effective in lowering the computational time needed for online parametrized simulations, we pay a huge amount of computational resources for the offline stage. We propose our methodology as a starting point which needs a deeper analysis and improvements in order to reach better performances: overcaming this issue will pave the way to applications of ROMs to more complicated \ocp s.}
\\ Moreover, we are moving towards nonlinear state equations in order to recover a complete optimal control model that could have impact in several fields of applications. Time dependent nonlinear \ocp s could be a way to simulate more and more realistic physical phenomena. Reduced order modelling, most of all in its nonlinear time dependent formulation, could actually be a suitable and versatile approach to be used, in order to drastically reduce the computational costs in real time contexts.
Another step forward could be adding a parameter stochastic dependency and uncertainty quantification in the optimality system.
For sure, another improvement is based on the development of proper error estimators. The extension of classical estimates for steady linear quadratic \ocp s to time dependent state equations will be the topic of future investigations.

\section*{Acknowledgements}
We acknowledge the support by European Union Funding for Research and Innovation -- Horizon 2020 Program -- in the framework of European Research Council Executive Agency: Consolidator Grant H2020 ERC CoG 2015 AROMA-CFD project 681447 ``Advanced Reduced Order Methods with Applications in Computational Fluid Dynamics''. We also acknowledge the INDAM-GNCS project ``Advanced intrusive and non-intrusive model order reduction techniques and applications''.
The computations in this work have been performed with RBniCS \cite{rbnics} library, developed at SISSA mathLab, which is an implementation in FEniCS \cite{fenics} of several reduced order modelling techniques; we acknowledge developers and contributors to both libraries.
\bibliographystyle{plain} 

\addcontentsline{toc}{chapter}{Bibliography}
\bibliography{BIB}

\end{document}